\documentclass{elsarticle}

\makeatletter
\def\ps@pprintTitle{%
 \let\@oddhead\@empty
 \let\@evenhead\@empty
 \def\@oddfoot{\centerline{\thepage}}%
 \let\@evenfoot\@oddfoot}
\makeatother
\usepackage{enumerate}
\usepackage{amsmath}
\usepackage{amssymb}  
\usepackage{listings}   
\usepackage{mathtools}
\usepackage{amsthm}
\usepackage[a4paper,top=4cm,bottom=4cm,left=4.5cm,right=4.5cm]{geometry}
\usepackage{multirow}
\usepackage{booktabs}
\usepackage{bigstrut}
\usepackage{hyperref}
\usepackage{tikz-cd}
\usepackage[all]{xy}
\usepackage{mathrsfs}
\usepackage{mathtools}
\usepackage{color}  
\usepackage{ulem}   

\usepackage{babel}
\usepackage{natbib}
\usepackage{url}
\makeatletter
\renewcommand{\@biblabel}[1]{[#1]\hfill}
\makeatother

\DeclareRobustCommand{\erase}{\bgroup\markoverwith{\textcolor{red}{\rule[.5ex]{2pt}{0.4pt}}}\ULon}

\newcommand{\wt}[1]{\widetilde{#1}}
\newcommand{\vv}[2]{v_{#1,#2,1},\dots,v_{#1,#2,m}}

\def\diag{\mathrm{diag}}
\def\wid{\mathrm{wid}} 

\def\sup{{\mathrm{sup}}}
\def\Kabn{{K_{[\alpha,\beta],n}}}

\def\End{{\mathrm{End}}}
\def\Im{{\mathrm{Im}}}
\def\coker{{\mathrm{coker}}}
\def\ee{e_1,\dots,e_m}
\def\Der{{\mathrm{Der}}}
\def\AA{\mathscr{A}}
\def\BB{\mathscr{B}}
\def\ZZ{{\mathbb Z}} 
\def\RR{{\mathbb R}}
\def\cc{{\mathscr{C}_\rho}}

\def\QQ{{\mathbb Q}}
\def\QP{{\mathbb{Q}_p}}

\def\Sp{{\mathscr{M}}}
\def\Spec{{\mathrm{Spec}}}

\def \ZP{{\mathbb{Z}_p}}

\def \GL{{\mathrm{GL}}}
\def \upd{{\overline{d}}}
\def \lowd{{\underline{d}}}
\def \Aut{{\mathrm{Aut}}}

\newtheorem{thm}{Theorem} [section]   
\newtheorem{cor}[thm] {Corollary}
\newtheorem{lem}[thm] {Lemma}
\newtheorem{prop}[thm] {Proposition}

\theoremstyle{definition}
\newtheorem{defn}[thm] {Definition}
\newtheorem{eg}[thm] {Example}
\theoremstyle{remark}
\newtheorem{rmk}[thm] {Remark}
\allowdisplaybreaks[4]

\begin{document}

\begin{frontmatter}
\pagenumbering{arabic}
\title{{\large\textbf{On generalized Fuchs theorem over $p$-adic polyannuli}}}
\author{Peiduo Wang}
\cortext[corresponding author]{Address: 153-8914 Komaba 3-8-1, Meguro, Tokyo, Japan}
\ead{wang-peiduo617@g.ecc.u-tokyo.ac.jp}
\address{Graduate School of Mathematical Sciences, the University of Tokyo}
\date{}
\begin{abstract}
In this paper, we study finite projective differential modules on $p$-adic polyannuli satisfying the Robba condition. Christol and Mebkhout proved the decomposition theorem (the $p$-adic Fuchs theorem) of such differential modules on one dimensional $p$-adic annuli under certain non-Liouvilleness assumption and Gachet generalized it to higher dimensional cases. On the other hand, Kedlaya proved a generalization of the $p$-adic Fuchs theorem in one dimensional case. We prove Kedlaya's generalized version of $p$-adic Fuchs theorem in higher dimensional cases.
\end{abstract}
\begin{keyword}
$p$-adic differential equations; $p$-adic Fuchs theorem; polyannuli;
\end{keyword}

\end{frontmatter}

\section*{Introduction}
Let $K$ be a complete nonarchimedean field of mixed characteristic $(0,p)$.
Christol and Mebkhout have given an intrinsic definition of the exponents of a finite free differential module on one dimensional annuli satisfying the Robba condition in \cite{CM2}. They have also shown that if the exponent has $p$-adic non-Liouville differences (\cite{Ked1} Definition 13.2.1), then there exists a canonical decomposition of this differential module into 
the ones with exponent identically equal to a single element. This is called the $p$-adic Fuchs theorem.
However, their work was found to be difficult due to the complicated nature of the Frobenius antecedent developed in \cite{CD}, on which their work was built. Dwork gave a simplified proof of $p$-adic Fuchs theorem on one dimensional annuli, in which Frobenius antecedent no more plays an important role. This method is also written in \cite{Ked1} with a slightly different way. After Dwork's proof on one dimensional annuli, Gachet proved the $p$-adic Fuchs theorem on higher dimensional polyannuli in \cite{Gac}.
The precise statement of this theorem is as follows:

\begin{thm}[Th\'eor\`eme in page 216 of \cite{Gac}]
    \label{0}
Let $P$ be a finite free differential module on an open polyannulus over $K$ for the derivations $t_i\partial_{t_i}$, with $1\leq i\leq n$ satisfying the Robba condition and admitting an exponent on some closed subpolyannulus of positive width with $p$-adic non-Liouville differences. Then $P$ adimits a basis on which the matrix of action of $\nabla(t_i\partial_{t_i})$ has entries in $K$ and its eigenvalues represent the exponent of $P$ for all $1\leq i\leq n$. Consequently, $P$ admits a canonical decomposition 
$$P=\bigoplus_{\lambda\in(\ZP/\ZZ)^n}P_\lambda,$$ 
in which each $P_\lambda$ has exponent identically equal to $\lambda$.
\end{thm}

Meanwhile, Kedlaya proved a generalized version of one dimensional $p$-adic Fuchs theorem, by loosing the condition on exponents from having $p$-adic non-Liouville differences to a weaker one, namely, having Liouville partition, and yet still gives a decomposition of such differential module. 
\begin{thm}[Theorem 3.4.22 in \cite{Ked2}]
\label{1}
Let $P$ be a finite free differential module satisfying the Robba condition over one dimensional annulus over $K$ associated to an open interval $I$. Let $J\subset I$ be a closed subinterval of positive width, and 
suppose that $P$ has an exponent $A$ over $J$ admitting a Liouville partition $\AA_1,\dots,\AA_k$(for definition see Definition 3.4.4 in\cite{Ked2}). Then there exists a unique direct sum decomposition $P_J=P_1\oplus\dots\oplus P_k$ such that for $g=1,\dots,k$, $P_g$ admits an exponent over $J$ weakly equivalent to $\AA_g$.
\end{thm}

Moreover, it is realized that the generalized $p$-adic Fuchs theorem implies the the original $p$-adic Fuchs theorem in one dimensional case.

In this paper, we prove a generalized version of higher dimensional $p$-adic Fuchs theorem. We define the notion of exponent $A$ for a finite projective differential module $P$  satisfying the Robba condition on higher dimensional polyannuli over $K$, and prove a decomposition theorem for $P$ with respect to a Liouville partition of $A$, which is similar to Theorem \ref{1}. It is worth mentioning that our result implies Theorem \ref{0}, and since our generalized $p$-adic Fuchs theorem work not only for finite free but also for finite projective differential modules, our result is possibly stronger than the result in \cite{Gac}. Also, though we basically follow the strategy developed by Kedlaya in \cite{Ked2}, there are new ingredients applied to get the decomposition from local ones, because of the lack of Quillen-Suslin theorem for arbitrary polyannuli over $K$. 

Let us explain the content of this paper. In Section $1$, we introduce several notations and basic properties that are frequently used later. Also we prove Quillen-Suslin theorem for polyannuli of width $0$ and of small positive width, which is essential for us to construct the category $\mathscr{C}_\rho$ in Section $2$. We also prove properties of abstract $p$-adic exponents in $\ZZ_p^n$, most of which are generalizations of those of abstract $p$-adic exponents in $\ZZ_p$ treated in Section 3.4 of \cite{Ked2}. We give the definition and prove several properties of $p$-adic non-Liouvilleness and Liouville partition, which are important for understanding our main theorem. 

In Section $2$, we introduce two categories $\mathscr{C}_\rho$ and $\mathscr{D}_\rho$ whose objects are modules over a certain direct limit ring with additional structures, and on which most of our work is based. We will prove that $\mathscr{D}_\rho$ is a subcategory of $\mathscr{C}_\rho$, and that $\mathscr{C}_\rho$ is an abelian category. More interestingly, all objects in $\mathscr{C}_\rho$ (and hence all objects in $\mathscr{D}_\rho$) are free.

In Section $3$, we define the notion of exponent for objects in $\mathscr{C}_\rho$, and prove the existence of it. Moreover, we prove that the exponent for an object in $\mathscr{C}_\rho$ is unique up to weak equivalence. Finally, we prove the generalized $p$-adic Fuchs theorem, which is our main goal.

Finally, the author would like to thank his supervisor Atsushi Shiho for introducing this problem, for offering helpful ideas, and for spending time discussing about drafts. The author is partially supported by FoPM(Forefront Physics and Mathematics Program to Drive Transformation) of the University of Tokyo.

\subsection*{Conventions}
\begin{enumerate}[(1)]
    \item Throughout the paper, we fix a complete nonarchimedean field $K$ of mixed characteristic $(0,p)$. $K^{\mathrm{alg}}$ denotes a fixed algebraic closure of $K$ and $k$ denotes the residue field of $K$. $|\cdot|$ denotes the norm on $K$ with $|p|=p^{-1}$, and let $\sqrt{|K^\times|}$ be the divisible closure of $|K^\times|\subset \RR_{>0}$. $\omega$ denotes the constant $p^{-1/(p-1)}$ in $\RR$. We fix a positive integer $n$ which will be the dimension of our polyannuli. $K[[t,t^{-1}]]$ denotes the set
$$ \{ f=\sum_{i_1,\dots,i_n\in \ZZ}f_{i_1 i_2\dots i_n}t_1^{i_1}t_2^{i_2}\dots t_n^{i_n}: f_{i_1 i_2\dots i_n}\in K \} $$
of formal Laurent series of $n$ variables over $K$.

\item For two $n$-tuples $\alpha=(\alpha_1,\dots,\alpha_n)$, $\beta=(\beta_1,\dots,\beta_n)$ in $\RR^n $, we write $\alpha<\beta$(resp. $\alpha>\beta$, $\alpha\leq \beta$, $\alpha\geq\beta$) if $\alpha_i<\beta_i$(resp. $\alpha_i>\beta_i$, $\alpha_i\leq \beta_i$, $\alpha_i\geq\beta_i$) for all $1\leq i\leq n$.
For an $n$-tuple of variables $t=(t_1,\dots,t_n)$ and $\alpha=(\alpha_1,\dots,\alpha_n)$, we denote $t_1^{\alpha_1}\dots t_n^{\alpha_n}$ by $t^\alpha$. For $i=(i_1,\dots,i_n)\in \ZZ^n$, $|i|$ denotes the sum of the absolute value of entries of $i$, namely, $|i|=|i_1|+\dots+|i_n|$.

\item A polysegment in $\RR^n_{>0}$ is a subset of the form $I=\prod_{i=1}^nI_i$, where each $I_i$ is a non-empty interval in $\RR_{>0}$. It is called open(resp. closed) if $I$ is an open(resp. closed) subset of $\RR_{>0}^n$. For $\alpha=(\alpha_1,\dots,\alpha_n)$, $\beta=(\beta_1,\dots,\beta_n)$ in $\RR^n_{>0}$, we denote the polysegment $\prod_{i=1}^n[\alpha_i,\beta_i]$ by $[\alpha,\beta]$.

\item For a positive integer $s$, $\Gamma_s$ denotes the group of $p^s$-th roots of unity in $K^{\mathrm{alg}}$, and $\Gamma=\bigcup_{s>0}\Gamma_s$. Also $\Gamma^n_s$ and $\Gamma^n$ denote the product of $n$ copies of $\Gamma_s$ and $\Gamma$, respectively. 

\item For a ring $R$ and a positive integer $r$, $\mathrm{M}_r(R)$ denotes the set of $r\times r$ matrices with entries in $R$.

\item For a commutative Banach algebra $A$, $\mathscr{M}(A)$ denotes the spectrum of $A$ in the sense of Berkovich.
\end{enumerate}

\section{Preliminaries}
\subsection{Notations and basic facts}

We denote the Berkovich affine $n$-space $(\Spec K[t_1,\dots,t_n])^{\mathrm{an}}$ over $K$ by $\mathbb{A}_K^n$  .
\begin{defn}
For a polysegment $I=\prod_{i=1}^n I_i\subset\mathbb{R}_{> 0}^n$, the polyannulus with radius $I$ over $K$ is the subspace of $\mathbb{A}_K^n$ defined by
$$\left\{ x\in \mathbb{A}_K^n: t_i(x)\in I_i ,\ 1\leq i\leq n \right\},$$
and we call such a subspace an open (resp. closed) polyannulus if $I$ is open (resp. closed). Moreover, we say that it is of positive width (resp. of width $0$) if each $I_i$ is not a point (resp. $I$ degenerates to a point).
The coordinate ring of this polyannulus is
$$\left \{ f=\sum_{i\in{\ZZ^n}}f_it^i\in K[[t,t^{-1}]]: \lim_{|i|\to \infty}|f_i|\rho^i=0\quad \forall \rho\in I\right \},$$ 
and we denote this ring by $K_{I,n}$. Here we put the subscript $n$ in the notation to emphasize the dimension of the associated polyannulus.
For $\rho\in I$, we define the $\rho$-Gauss norm of $f=\sum_{i\in \ZZ^n}f_i t^i\in K_{I,n}$ to be $|f|_\rho:=\max_{i\in\ZZ^n} |f_i|\rho^i$.
\end{defn}

When $I$ is closed, $K_{I,n}$ is a $K$-affinoid algebra in the sense of Berkovich, and the supremum norm (which is power multiplicative but not necessarily multiplicative) on $K_{I,n}$ is defined by $|f|_{I}:=\max_{\rho\in I}\{|f|_\rho\}$. For polysegments $J\subset I$ in $\RR^n_{>0}$ and a $K_{I,n}$-module $P$, we denote the module $K_{J,n}\otimes_{K_{I,n}} P$ by $P_J$. Also for $\rho\in I$, we denote the module $K_{[\rho,\rho],n}\otimes_{K_{I,n}} P$ by $P_{0,\rho}$.

\begin{defn}\label{above7}
Let $I$ be a polysegment in $\RR^n_{>0}$ and $\rho\in I$. The $\rho$-width vector of $f=\sum_{i\in \ZZ^n}f_it^i\in K_{I,n}$ is defined to be the difference of the maximal and the minimal element(with respect to the lexicographic order) in the set $\{ i\in \ZZ^n:|f_i|\rho^i=|f|_\rho \}$. We denote the $\rho$-width vector of $f$ by $\wid_\rho(f)$.
    
\end{defn}
    
\begin{lem}
    Let $I,\rho$ be as in Definition \ref{above7} and let $P,Q\in K_{I,n}$. Then $\wid_\rho(PQ)=\wid_\rho(P)+\wid_\rho(Q)$.
\end{lem}
    
\begin{proof}
    Write $P=\sum_{i\in\ZZ^n}P_it^i$, $Q=\sum_{i\in\ZZ^n}Q_it^i$, $PQ=\sum_{i\in\ZZ^n}(PQ)_{i}t^i$. Then we have
    $$ (PQ)_i=\sum_{j+k=i}P_jQ_k .$$
    For $S\in\{P,Q\}$, let $M_S$ and $m_S$ be the maximal and the minimal index $i=(i_1,\dots,i_n)$(with respect to the lexicographical order), respectively, such that $|S_i|\rho^i=|S|_\rho$.
    Then, for $j,k\in\ZZ^n$ with $j+k=i$, we have the inequality $|P_jQ_k|\rho^i\leq|P|_\rho|Q|_\rho$, and the equality can occur only when $m_P\leq j\leq M_P$ and $m_Q\leq k\leq M_Q$.

    Then, if $i>M_P+M_Q$ or $i< m_P+m_Q$, we have
    $$ |(PQ)_i|_\rho\rho^i\leq \max_{j+k=i}|P_jQ_k|\rho^i<|P|_\rho |Q|_\rho. $$

    Also, if $i=M_P+M_Q$(resp. $i=m_P+m_Q$), there is a unique summand $P_{M_P}Q_{M_Q}$(resp. $P_{m_P}Q_{m_Q}$) of $(PQ)_i$ such that $|P_{M_P}Q_{M_Q}|\rho^i=|P|_\rho|Q|_\rho$(resp. $P_{m_P}Q_{m_Q}\rho^i=|P|_\rho|Q|_\rho$). These imply the desired assertion. 
\end{proof}
    
\begin{cor}{\label{invt}}
    Let $f=\sum_{i\in \ZZ^n} f_it^i\in \Kabn$. Then $f$ is a unit in $\Kabn$ if and only if $wid_\rho(f)=0\in \ZZ^n$ for all $\rho \in [\alpha,\beta]$, that is, for all $\rho \in [\alpha,\beta]$, there exists $i\in \ZZ^n$ such that $|f-f_it^i|_\rho<|f|_\rho$.
\end{cor}
    
\begin{proof}
    If $|f-f_it^i|_\rho<|f|_\rho$, then we can write $f=f_it^{i}(1-(1-f_i^{-1}t^{-i}f))$, thus
    $$ f^{-1}=f_i^{-1}t^{-i}\sum_{n=0}^\infty(1-f_i^{-1}t^{-i}f)^n, $$
    with the series on the right hand side converges for all $\rho$-Gauss norms($\rho\in[\alpha,\beta]$).
    
    If $f$ is a unit in $\Kabn$, then we have
    $$ 0=\wid_\rho(1)=\wid_\rho(ff^{-1})=\wid_\rho(f)+\wid_\rho(f^{-1})\geq 0, $$
    and this implies that $\wid_\rho(f)=0$ for all $\rho$-Gauss norms($\rho\in[\alpha,\beta]$).
\end{proof}

\begin{cor}
    Let $f\in\Kabn$ . If $f$ is a unit in $K_{[\rho,\rho],n}$ for some $\rho\in (\alpha,\beta)$, then there exists $\delta$ and $\gamma$ with $\alpha<\delta<\rho<\gamma<\beta$ such that  $f$ is a unit in $K_{[\delta,\gamma],n}$.
\end{cor}
\begin{proof}
    Write $f=\sum_{i\in\ZZ^n}f_it^i$. By Corollary \ref{invt}, there exists an index $i$ such that $|f-f_it^i|_\rho<|f|_\rho$. Since functions $\sigma\mapsto|f-f_it^i|_{\sigma}$, $\sigma\mapsto|f|_\sigma$ are continuous, we can take $\alpha<\gamma<\rho<\delta<\beta$ such that $|f-f_it^i|_\sigma<|f|_\sigma$ for all $\sigma\in[\gamma,\delta]$. So $f$ is a unit in $K_{[\gamma,\delta],n}$.
\end{proof}

\begin{defn}[cf. \cite{Ked1}, Definition 8.1.5]
Let $R$ be a Banach $K$-algebra with norm $|\cdot|$, and let $\alpha,\beta\in\RR_{>0}$ with $\alpha<\beta$. We define the ring $R[[_0 \alpha/t,t/\beta\rangle$ (the ring of Laurent series which converges and take bounded values on the semi-open polyannulus $\alpha<|t|\leq \beta$) as follows:
$$ R[[_0 \alpha/t,t/\beta\rangle := \left\{ \sum_{i\in\ZZ}c_it^i: c_i\in R,\ \sup_i \{|c_i|\alpha^i\}<\infty,\ \lim_{i\to\infty}|c_i|\beta^i=0  \right\}. $$ 
\end{defn}

\subsection{Quillen-Suslin theorem over polyannuli of width $0$}
It was shown in Satz 3 in p.131 of \cite{Lut} that all finite projective modules over the polyannulus with inner and outer radii $1$ are free. We generalize this result to any polyannuli of width $0$, basically following the argument in \cite{KedF}, Section 6.

From now on in this subsection, unless otherwise specified, let \break$\rho=(\rho_1,\dots,\rho_n)\in \sqrt{|K^\times|}^n$, and let $R=K_{[\hat{\rho},\hat{\rho}],n-1}$, where $\hat{\rho}=(\rho_1,\dots,\rho_{n-1})$. Then $R$ admits the $\hat{\rho}$-Gauss norm(denoted by $|\cdot|$) and $g=\sum_{i\in\ZZ} g_i t_n^i\in R[[_0 \rho/t,t/\rho\rangle$ admits a natural norm $|g|=\sup_i\{ |g_i| \rho_n^i \}$.
\begin{defn}
An element $g=\sum_{i\in\ZZ} g_i t_n^i$ in $ R[[_0 \rho_n/t_n,t_n/\rho_n\rangle$ is called $t_n$-bidistinguished of (upper and lower) degree $(\lowd,\upd)$ if the following conditions are satisfied:
\begin{enumerate}[(1)]
    \item  Both $g_\lowd$ and $g_\upd$ are units in $R$.
    \item $|g_\lowd|\rho_n^\lowd=|g_\upd|\rho_n^\upd=|g|$, $|g_\lowd|\rho_n^\lowd>|g_\nu|\rho_n^\nu$ for all $\nu<\lowd$  and $|g_\upd|\rho_n^\upd>|g_\mu|\rho_n^\mu$ for all $\mu>\upd$.
\end{enumerate}
\end{defn}

Then we apply Theorem 2.2.2 of \cite{Ked1} to deduce the following argument:

\begin{prop}
\label{Factorization1}
Let $f=\sum_{i\in\ZZ}f_it_n^i\in  R[[_0 \rho_n/t_n,t_n/\rho_n\rangle$($f_i\in R$) be a $t_n$-bidistinguished element of degree $(\lowd,\upd)$. Then there exists a unique factorization $f=f_\upd t_n^\upd gh$ with 
$$ g\in  R[[_0 \rho_n/t_n,t_n/\rho_n\rangle\cap R[[t_n]],\ \ h\in R[[_0 \rho_n/t_n,t_n/\rho_n\rangle \cap R[[t_n^{-1}]] $$
and $|h|=|h_0|=1$, $|g-1|<1$, where $h_0\in R$ is the constant term of $h$.
\end{prop}

\begin{proof}
Without loss of generality, we may consider $f/f_\upd t_n^\upd$ instead of $f$, and thus we can assume that $|f|=|f_0|=1$. let
\begin{align*}
U &= \left\{ g\in  R[[_0 \rho_n/t_n,t_n/\rho_n\rangle\cap R[[t_n]]:|g|\leq 1 \text{ and $g$ has no constant term} \right\},\\
V &= \left\{ h\in  R[[_0 \rho_n/t_n,t_n/\rho_n\rangle\cap R[[t_n^{-1}]]:|h|\leq 1 \right\},\\
W &= \left\{ \varphi\in  R[[_0 \rho_n/t_n,t_n/\rho_n\rangle: |\varphi|\leq 1 \right\}.
\end{align*}

Also, for $\varphi=\sum_{i\in \ZZ}\varphi_i t_n^i\in R[[_0 \rho_n/t_n,t_n/\rho_n\rangle$, define $\varphi_{>0}=\sum_{i=1}^\infty \varphi_i t_n^i$, $\varphi_{\geq 0}=\sum_{i=0}^\infty \varphi_i t_n^i$, $\varphi_{\leq 0}=\sum_{i=-\infty}^0 \varphi_i t_n^i$, and let $a=1$, $b=f_{\leq 0}$, $c=f$, $\lambda=1$. With these notations, we check the following conditions (a),(b),(c),(d) in Theorem 2.2.2 of \cite{Ked1}:
\begin{enumerate}[(a)]
    \item $U,V$ are complete and $UV\subset W$.
    \item The map $U\times V\to W$: $(u,v)\mapsto av+ub$ is surjective.
    \item $|av+ub|\geq \lambda\max\{ |a||v|,|b||u| \}$.
    \item $ab-c\in W$ and $|ab-c|<\lambda^2|c|$.
\end{enumerate}

The condition (a) is obvious. To prove (b), note that, since $|f|=1$ and $f$ is $t_n$-bidistinguished of upper degree $0$, $b=f_{\leq 0}$ is a unit with $|b|=1$. Then, for $\varphi\in W$, if we put $u=(b^{-1}\varphi)_{>0}\in U$ and $v=b(b^{-1}\varphi)_{\leq 0}\in V$, we have 
\begin{align*}
    \varphi & = b(b^{-1}\varphi)=b(b^{-1}\varphi)_{\leq 0}+b(b^{-1}\varphi)_{>0} \\
    & =v+ub=av+ub
\end{align*}
So the condition (b) is proved.

Next, for $\varphi=v+ub$ with $u\in U$, $v\in V$, $b^{-1}\varphi=b^{-1}v+u$ with $b^{-1}v\in V$. Then, by definition of $U,V$, we have 
$$ |\varphi|=|b^{-1}\varphi|=\max\{ |b^{-1}v|,|u| \}=\max\{ |v|,|b||u| \}. $$
So the condition (c) is proved.

Finally, since $f$ is $t_n$-bidistinguished of upper degree $0$,
$$|ab-c|=|f_{\leq 0}-f|=|f_{>0}|<|f|=|c|,$$
and so the condition (d) is proved. Then by Theorem 2.2.2 of \cite{Ked1}, we have the desired conclusion.
\end{proof}

\begin{lem}
    \label{multiply_small_x}
Let $f\in K_{[\rho,\rho],n} $ be $t_n$-bidistintuished of degree $(\lowd,\upd)$, and $x\in K_{[\rho,\rho],n}$ be an element with $|x|<1$. Then $(1+x)f$ is also $t_n$-bidistinguished of degree $(\lowd,\upd)$.
\end{lem}
\begin{proof}
Write $x=\sum_{i\in\ZZ}x_it_n^i$ with $x_i\in R$. Then $|x_i|\rho_n^i<1$ for all $i$, and we have
\begin{align*}
(1+x)f & = \sum_{r\in\ZZ} f_r t_n^r +\sum_{i\in\ZZ} x_it_n^i\sum_{j\in\ZZ} f_jt_n^j\\
& = \sum_{r\in\ZZ}f_r t_n^r+\sum_{r\in\ZZ}(\sum_{i+j=r}x_if_j)t_n^r \\
&= \sum_{r\in \ZZ}(f_r+\sum_{i+j=r}x_if_j)t_n^r.    
\end{align*}
Since $f$ is $t_n$-bidistinguished of degree $(\lowd,\upd)$, we have $|f_\lowd|=|f|\rho_n^{-\lowd}$ and $|f_\upd|=|f|\rho_n^{-\upd}$. Since $|x_i|<\rho_n^{-i}$ and $|f_j|\leq |f|\rho_n^{-j}$,
$$ |\sum_{i+j=d}x_i f_j|<\max_{i+j=r}\{|f|\rho_n^{-(i+j)}\}=|f|\rho_n^{-r} .$$
If $r\in \{ \lowd,\upd \}$, since $f_r$ is a unit and $|f_r|=|f|\rho_n^{-r}>|\sum_{i+j=r}x_i f_j|$, we conclude that $f_r+\sum_{i+j=r}x_i f_j$ is also a unit whose norm is equal to $|f_r|$. For $r<\lowd$ or $r>\upd$,
$$ |f_r+\sum_{i+j=r}x_i f_j|\leq \max\{ |f_r|,\max_{i+j=r}|x_if_j|\}< |f|\rho_n^{-r}. $$
Thus $(1+x)f$ is $t_n$-bidistinguished of degree $(\lowd,\upd)$.
\end{proof}

Now we can prove the following factorization result in $K_{[\rho,\rho],n}$.

\begin{prop}\label{weierstrass}
For a $t_n$-bidistinguished element $f\in K_{[\rho,\rho],n}$ of degree $(\lowd,\upd)$, there exists a monic polynomial $P\in R[t_n]$ and a unit $u\in K_{[\rho,\rho],n}^\times$ such that $f=Pu$.
\end{prop}
\begin{proof}
By Proposition\ref{Factorization1}, we can factor $f$ in $ R[[_0 \rho_n^{-1}/t_n^{-1},t_n^{-1}/\rho_n^{-1}\rangle$ as $f=f_\lowd t_n^\lowd gh$ with 
$$ f_\lowd\in R^\times,\ g\in R[[t_n^{-1}]]\cap R[[_0 \rho_n^{-1}/t_n^{-1},t_n^{-1}/\rho_n^{-1}\rangle,\ h\in R[[t_n]]\cap R[[_0 \rho_n^{-1}/t_n^{-1},t_n^{-1}/\rho_n^{-1}\rangle, $$
$|g-1|<1$ and $|h|=|h_0|=1$, where $h_0\in R$ is the constant term of $h$. In particular, $g$ is a unit in $R\langle \rho_n/t_n \rangle$ and hence 
$$ g^{-1}f\in R\langle \rho_n/t_n,t_n/\rho_n\rangle\cap t_n^\lowd R[[t_n]] \subset t_n^\lowd R\langle t_n/\rho_n \rangle .$$
Since $|g-1|<1$, $|g^{-1}-1|<1$ and so $g^{-1}=1+y$ with $y\in R\langle \rho_n/t_n \rangle$ and $|y|<1$. 
Thus, by Lemma \ref{multiply_small_x}, $g^{-1}f=(1+y)f$ is also $t_n$-bidistinguished of degree $(\lowd,\upd)$. So $g^{-1}f/t_n^\lowd\in R\langle t_n/\rho_n \rangle$ is $t_n$-bidistinguished of degree $(0,\upd-\lowd)$. Then by Lemma 6.1 of \cite{KedF}, we can write $g^{-1}f$ as $t_n^\lowd e P$, where $e$ is a unit in $R\langle t_n/\rho_n \rangle$ (hence a unit in $K_{[\rho,\rho],n}$) and $P\in R[t_n]$ is a monic polynomial of degree $\upd-\lowd$. Thus $f=(get_n^\lowd)P$ gives the desired factorization.
\end{proof}

\begin{lem}
    \label{change_to_bid}
For any non-zero element $f\in K_{[\rho,\rho],n}$, there exists an isometric isomorphism $\sigma\in \Aut(K_{[\rho,\rho],n})$ such that $\sigma(f)$ is $t_n$-bidistinguished.
\end{lem}
\begin{proof}
Let $f=\sum_{i\in \ZZ^n} f_it^i$ and $u=(u_1,\dots,u_n), v=(v_1,\dots,v_n)$ be the largest and least index $\mu$ with respect to the lexicographical order in $\ZZ^n$ such that $|f_\mu|\rho^\mu=|f|$. Let $L$ be a positive integer with $L\geq \max_{1\leq i\leq n}|\mu_i|$ for any $\mu=(\mu_1,\dots,\mu_n)$ satisfying $|f_\mu|\rho^\mu=|f|$. Take $w\in K$ such that $|w|\rho_n^s=1$ for some $s\in\ZZ_{>0}$. Let 
\begin{align*}
\sigma_j: K_{[\rho,\rho],n} & \longrightarrow K_{[\rho,\rho],n}\ \ \text{ be the homomorphism defined by}\\
t_i & \longmapsto t_i(wt_n^s)^{j^{n-i}}\ \text{for}\ 1\leq i\leq n-1,\\
t_n & \longmapsto t_n.
\end{align*}
$\sigma_j$ is obviously an isometric automorphism. We claim that $\sigma_j(f)$ is $t_n$-bidistinguished of degree $(\lowd,\upd)$ where $\lowd=\sum_{i=1}^{n-1}sj^{n-i}v_i+v_n$ and $\upd=\sum_{i=1}^{n-1}sj^{n-i}u_i+u_n$, if $j>3L$.

For all $\mu=(\mu_1,\dots,\mu_n)$ with $|f_\mu|\rho^\mu=|f|$ and $\mu\neq u$, $\mu\neq v$, the degree of $\sigma_j(f_\mu t^\mu)$ with respect to $t_n$ is 
$$ \sum_{i=1}^{n-1}sj^{n-i}\mu_i+\mu_n. $$
By definition of $u$ and $v$, there exist $1\leq p,q\leq n$ such that 
$$\mu_1=u_1,\ \dots,\ \mu_{p-1}=u_{p-1},\ \mu_p>u_p,$$
and 
$$ \mu_1=v_1,\ \dots,\ \mu_{q-1}=v_{q-1},\ \mu_q<v_q. $$
Thus, if $j>3L$, we have
\begin{align*}
\sum_{i=1}^{n-1}sj^{n-i}\mu_i+\mu_n & \geq \sum_{i=1}^{p-1} sj^{n-i}u_i+sj^{n-p}(u_p+1)+\sum_{i=p+1}^{n-1} sj^{n-1}\mu_i +\mu_n\\
& =\sum_{i=1}^{n-1}sj^{n-i}u_i+u_n+sj^{n-p}+\sum_{i=p+1}^{n-1}sj^{n-i}(\mu_i-u_i)+
(\mu_n-u_n)\\
& \geq \sum_{i=1}^{n-1}sj^{n-i}u_i+u_n+s( j^{n-p}- 2L\sum_{i=p+1}^{n-1}j^{n-i} -2L )\\
& > \sum_{i=1}^{n-1}sj^{n-i}u_i+u_n
\end{align*}
Here we use the inequality
$$ j^{n-p}- 2L\sum_{i=p+1}^{n-1}j^{n-i} -2L=\frac{(j-1-2L)j^{n-p}+2L}{j-1}\geq 0 $$
when $j>3L$.
A similar computation yields that $\sum_{i=1}^{n-1}sj^{n-i}\mu_i+\mu_n< \sum_{i=1}^{n-1}sj^{n-i}v_i+v_n$. 

Now write $\sigma_j(f)=g=\sum_{i\in\ZZ}g_it_n^i$, Then, by the computation above, the term $\sigma_j(f_\mu t^\mu)$ with $|f_\mu|\rho^\mu=|f|$, $\mu\neq u$ does not contribute to the term $g_\upd t_n^\upd$. So we can write 
$$g_\upd=f_u t_1^{u_1}\dots t_{n-1}^{u_{n-1}}+g'_\upd$$ 
with $g'_\upd \in R$ and $|g'_\upd|<|f_u|\rho_1^{u_1}\cdots\rho_{n-1}^{u_{n-1}}$. Thus we conclude that $g_\upd$ is a unit and $|g_\upd|\rho_n^\upd=|f|$. 
By the same way we can conclude that this is also true for $g_\lowd$. By a similar argument, for $\mu<\lowd$ or $\mu>\upd$, we have $|g_\mu|\rho_n^\mu<|f|$.
\end{proof}
\begin{rmk}\label{uniform_of_j}
The proof of Lemma \ref{change_to_bid} implies that, for finitely many elements $f_1,\dots,f_m\in K_{[\rho,\rho],n}$, we can choose a uniform $j$ such that $\sigma_j(f_i)$ are all $t_n$-bidistinguished. 
\end{rmk}
\begin{defn}
An $r$-tuple $f=(f_1,\dots,f_r)$ of elements of a ring $R$ is called unimodular if these elements generate the unit ideal. Two $r$-tuples $f$ and $g$ are called equivalent and denoted by $f\sim g$ if there is a matrix $M\in\GL_r(R)$ such that $Mf^T=g^T$.
\end{defn}

\begin{prop}[cf. \cite{KedF}, Proposition 6.4]\label{unimodular}
Let $f=(f_1,\dots,f_r)$ be a unimodular tuple in $K_{[\rho,\rho],n}$. Then $f\sim e_1=(1,0,\dots,0)$.
\end{prop}
\begin{proof}
We prove it by induction on $n$. The case $n=0$ is clear. Suppose that $n>0$. By Lemma \ref{change_to_bid} and Remark \ref{uniform_of_j}, we can choose $\sigma\in \Aut(K_{[\rho,\rho],n})$ which is also an isometry such that $f_1,\dots,f_r$ are all $t_n$-bidistinguished, and hence by Proposition \ref{weierstrass}, each $\sigma(f_i)$ can be factorized as $e_iP_i$ where $e_i\in K_{[\rho,\rho],n}^\times$ and $P_i$ is a monic polynomial in $R[t_n]$. Since 
$$\sigma(f)=\diag(e_1,\dots,e_r)(P_1,\dots,P_r)^T,$$
$\sigma(f)\sim P:=(P_1,\dots,P_r)^T$. By \cite{La}, XXI, Theorem 3.4, $P\sim P(0)$, which is an $r$-tuple with entries in $K_{[\hat{\rho},\hat{\rho}],n-1}$, and by induction hypothesis, $P(0)\sim e_1$. Thus there exists $M\in\GL_r(K_{[\rho,\rho],n})$ such that $\sigma(f)=Me_1$. By applying $\sigma^{-1}$ on both sides, we deduce that $f\sim e_1$.
\end{proof}

\begin{prop}[cf. \cite{KedF}, Proposition 6.5]\label{f.f_resolution}
Every finite module over $R=K_{[\rho,\rho],n}$ has a finite free resolution.
\end{prop}
\begin{proof}
We proceed by induction on $n$. There is nothing to prove if $n=0$. Let $M$ be a finite module over $R$. Since $R$ is Noetherian (by \cite{BGR}, 5.2.6.1), $M$ is also Noetherian. By \cite{La}, X, Corollary 2.8, there exists a finite filtration
$$ M=M_1\supset M_2\supset\dots\supset M_n=0, $$
such that $M_i/M_{i+1}\simeq R/\mathfrak{p_i} $ for some $\mathfrak{p_i}\in \Spec R$. Meanwhile, by \cite{La}, XXI, Theorem 2.7, to prove that $M$ has a finite free resolution, it suffices to prove that for an exact sequence of $R$-modules
$$0\to N\to M\to P\to 0,$$
both $N$ and $P$ have finite free resolutions. Because we have the exact sequences 
$$ 0\to M_{i+1}\to M_i\to R/\mathfrak{p}_i\to 0,\ \quad i=1,\dots,n-1,  $$
we are reduced to proving that $R/\mathfrak{p}$ admits a finite free resolution for all $\mathfrak{p}\in\Spec R$. Moreover, by the exact sequence
$$ 0\to \mathfrak{p}\to R\to R/\mathfrak{p}\to 0, $$
it suffices to prove that $\mathfrak{p}$ admits a finite free resolution for all $\mathfrak{p}\in \Spec R$.

Suppose that $f_1,\dots,f_m\in R$ generates $\mathfrak{p}$. By Lemma \ref{change_to_bid} and Remark \ref{uniform_of_j}, there exists an automorphism $\sigma$ such that $\sigma(f_i)$ is $t_n$-bidistinguished for $1\leq i\leq m$. Then, by Proposition \ref{weierstrass}, we can write $\sigma(f_i)=e_iP_i$, where $e_i\in K_{[\rho,\rho],n}^\times$ and $P_i$ is a monic polynomial in $R[t_n]$. So we have a $\sigma$-linear isomorphism $\mathfrak{p}\cong J\otimes_{R[t_n]} K_{[\rho,\rho],n}$ with $J=(P_1,\dots,P_m)$. By \cite{La}, XXI, Theorem 2.8, $J$ (and hence $\mathfrak{p}$) has finite free resolution and by induction hypothesis, $M$ admits a finite free resolution.
\end{proof}

\begin{thm}[cf. \cite{KedF}, Proposition 6.6]
    \label{QS0}
For $\rho\in \sqrt{|K^\times|}^n$, every finite projective module over $K_{[\rho,\rho],n}$ is free.
\end{thm}
\begin{proof}
Let $P$ be a finite projective module over $R$. By Proposition \ref{f.f_resolution}, $P$ admits a finite free resolution and hence is stably free by \cite{La}, XXI, Theorem 2.1. On the other hand, Proposition \ref{unimodular} implies that $P$ has the unimodular column extension property(see p.849, p.846 of \cite{La} for definition). 
Then by \cite{La}, XXI, Theorem 3.6, we conclude that $P$ is free.
\end{proof}

Now let $\rho=(\rho_1,\dots,\rho_n)\in \RR^n_{>0}$. Let $V$ be the $\QQ$-linear subspace of $\RR/\log\sqrt{|K^\times|}$ generated by $\log\rho_1,\dots,\log\rho_n$. By properly changing the indices, we may assume that $\log\rho_1,\dots,\log\rho_l$ is a basis of $V$. Put $\rho'=(\rho_1,\dots,\rho_l)$ and $\rho''=(\rho_{l+1},\dots,\rho_n)$. Then $K'=K_{[\rho',\rho'],l}$ is a field and the norm on $K'$ extends that on $K$. Thus $\rho''\in \sqrt{|K'^\times|}^{n-l}$ and $K_{[\rho,\rho],n}=K'_{[\rho'',\rho''],n-l}$. So, by Theorem \ref{QS0}, we obtain the following result:

\begin{cor}\label{QS00}
For $\rho\in \RR^n_{>0}$, every finite projective module over $K_{[\rho,\rho],n}$ is free.
\end{cor}

\subsection{Quillen-Suslin theorem over polyannuli of small positive width }
Now we generalize Corollary \ref{QS00} to polyannuli of small positive width, and obtain the following theorem.

\begin{thm} 
Let $\alpha$, $\beta\in \RR_{>0}^n$ with $\alpha<\beta$, and let $P$ be a finite projective $\Kabn$ module. Then, for any $\rho\in(\alpha,\beta)$, there exist $\alpha'$ and $\beta'$ with $\alpha<\alpha'<\rho<\beta'<\beta$ such that $P_{[\alpha',\beta']}$ is free.
\label{QS}
\end{thm}
\begin{proof}
$P_{0,\rho}$ is free by Corollary \ref{QS00}, and we can take a free basis $e_1,\dots,e_m$ of it.
Take $\epsilon\in (0,1)$, and let $$D_j:=\left\{ \sum_{i=1}^m a_ie_i:{|a_j-1|}_{\rho}<\epsilon, {|a_k|}_{\rho}\leq\epsilon\ \text{for all}\ k\neq j \right\}.$$
 Since $P$ is dense in $P_{0,\rho}$ with respect to the supremum norm associated to the $\rho$-Gauss norm and the basis $e_1,\dots,e_m$, there exists an element $e'_j\in P\cap D_j$ for each $1\leq j\leq m$. Then the representation matrix of the map $e_j\mapsto e'_j$ with respect to the basis $e_1,\dots,e_m$ can be written as $I+X$ with $|X|_\rho<\epsilon$ and hence an element in $\GL_r(K_{[\rho,\rho],n})$. This implies that $e'_1,\dots,e'_m$ is also a basis of $P_{0,\rho}$. 

 We prove that $e'_1,\dots,e'_m$ is also a basis for $P_{[\alpha',\beta']}$ for some $\alpha<\alpha'<\rho<\beta'<\beta$. Take a finite affinoid covering $\Sp(\Kabn)=\bigcup_{l=1}^N \Sp (A_l)$ such that $P\otimes_\Kabn A_l$ is free. For a fixed $l$, take a basis $f_1,\dots,f_m$ of $P\otimes_\Kabn A_l$, let $B_l$ be the representation matrix of the map $f_i\mapsto e_i'$ with respect to the basis $f_1,\dots, f_m$ and put $b_l=\det B_l\in A_l$.
 Since $e'_1,\dots,e'_m$ is a basis of $P_{0,\rho}$, $b_l$ is invertible in the coordinate ring of affinoid subdomain $ \Sp (A_l)\cap\bigcap_{j=1}^n\{|t_j|=\rho_j\}$ and so
  $$ \{x\in \Sp (A_l):b_l(x)=0 \}\cap\bigcap_{j=1}^n\{|t_j|=\rho_j\}=\emptyset. $$

 We prove the claim that, for $i$ with $1\leq i\leq n$, there exist $\alpha'_j$ and $\beta'_j$ for $j\geq i$ such that 
 $$\{x\in \Sp (A_l):b_l(x)=0 \}\cap\bigcap_{j=1}^{i-1}\{|t_j|=\rho_j\}\cap\bigcap_{j=i}^n\{\alpha'_j\leq|t_i|\leq\beta'_j\}=\emptyset $$
by descending induction on $i$. By induction hypothesis, we may assume that $$\{x\in \Sp (A_l):b_l(x)=0 \}\cap\bigcap_{j=1}^{i}\{|t_j|=\rho_j\}\cap\bigcap_{j=i+1}^n\{\alpha'_j\leq|t_i|\leq\beta'_j\}=\emptyset .$$
Since $t_i$ is a continuous function, it attains the maximum $\alpha''_i$ on the affinoid subdomain
$$\{x\in \Sp (A_l):b_l(x)=0 \}\cap\bigcap_{j=1}^{i-1}\{|t_j|=\rho_j\}\cap\{|t_i|\leq \rho_i\}\cap\bigcap_{j=i+1}^n\{\alpha'_j\leq|t_i|\leq\beta'_j\} ,$$
which is Hausdorff and compact by \cite{Ber} Theorem 1.2.1 unless it is empty(when it is empty, put $\alpha''_i=\alpha_i$). Then, $\alpha''_i\neq \rho_i$ because otherwise it will contradict the induction hypothesis.
By the same reason we may choose $\beta''_i>\rho_i$ as the minimum of the function $t_i$ on the affinoid subdomain
$$\{x\in \Sp (A_l):b_l(x)=0 \}\cap\bigcap_{j=1}^{i-1}\{|t_j|=\rho_j\}\cap\{|t_i|\geq \rho_i\}\cap\bigcap_{j=i+1}^n\{\alpha'_j\leq|t_i|\leq\beta'_j\} $$
unless it is empty(when it is empty, put $\beta''_i=\beta_i$). Then we take $\alpha'_i$ and $\beta'_i$ such that $\alpha''_i<\alpha'_i<\rho_i$ and $\rho_i<\beta'_i<\beta''_i$ and finish the induction.

By this claim, we conclude that 
$$\{x\in \Sp (A_l):b_l(x)=0 \}\cap\bigcap_{j=1}^n\{\alpha'_j\leq|t_i|\leq\beta'_j\}=\emptyset ,$$
and so $b_l$ is invertible in the coordinate ring of the affinoid subdomain 
$$ \Sp(A_l)\cap\bigcap_{i=1}^n\{ \alpha'_i\leq |t_i|\leq \beta'_i \}. $$
Consider the homomorphism $K_{[\alpha',\beta'],n}^{\oplus r}\to P_{[\alpha',\beta']} $ which sends the canonical basis of $K_{[\alpha',\beta'],n}^{\oplus r}$ to $e'_1,\dots,e'_m$. Since this is an isomorphism on each member of the affinoid covering 
$$ \bigcap_{i=1}^n\{ \alpha'_i\leq |t_i|\leq \beta'_i \}= \bigcup_{l=1}^N\left( \Sp (A_l)\cap\bigcap_{i=1}^n\{\alpha'_i\leq|t_i|\leq\beta'_i\} \right),$$ this is itself an isomorphism.
\end{proof}

\subsection{Abstract $p$-adic exponents}
For $x\in \QP$, we denote by $\langle x\rangle$ the smallest positive rational number $a$ such that one of $x-a$ and $x+a$ is a $p$-adic integer.

We state some basic facts about the symbol $\langle\ \cdot\ \rangle$.

\begin{lem}
\begin{enumerate}[(1)]
\item For any $x$,$y\in\QP$, $\langle x+y\rangle\leq \langle x\rangle+\langle y\rangle$.
\item $\langle x\rangle =\langle -x\rangle $.
\item $p\langle x\rangle\geq \langle px\rangle$.
\end{enumerate}
\end{lem}
\begin{proof}
For (1), take $a$ (resp.$b$)$\in \QQ_{>0}$ such that one of $x\pm a$ (resp. $y\pm b$) belongs to $\ZP$. We may assume $a\leq b$. If $x+a$, $y+b$ or $x-a$, $y-b$ are in $\ZP$,
$$(x+y)\pm(a+b)=(x\pm a)+(y\pm b)\in \ZP.$$
And if $x-a$, $y+b$ or $x-a$, $y+b$ are in $\ZP$,
$$(x+y)\pm(b-a)=(x\mp a)+(y\pm b)\in \ZP.$$

(2) is simply because $x\pm b\in \ZP$ if and only if $-x\mp b\in \ZP$.

For (3), let $b\in \QQ_{>0}$ such that one of $x\pm b$ belongs to $\ZP$, then one of $px\pm pb$ belongs to $\ZP$.
\end{proof}

\begin{defn}
We say that $a\in \ZP$ is a $p$-adic Liouville number if $a\notin\ZZ$ and 
$$ \liminf_{m\to \infty} \frac{p^m}{m}\left\langle\frac{a}{p^m}\right\rangle<\infty.$$
If $a$ is not $p$-adic Liouville, we say that it is a $p$-adic non-Liouville number.
\end{defn}

In the following, for a multisubset $A=\{ A_i,\dots,A_m \}$ of $\ZZ_p^n$, we denote the $i$-th entry of $A_j$ by $A_j^i$, and denote the multisubset $\{ A_1^i,\dots,A_m^i \}$ of $\ZZ_p$ by $A^i$.

\begin{defn}[D\'efinitions in p.194 of \cite{Gac}, cf. \cite{Ked2}, Definition 3.4.2] 
Let $A=\{A_1,\dots,A_m\}$ be a multisubset of $\ZZ_p^n$. We say that $A$ is $p$-adic non-Liouville in the $r$-th direction if $A_{j}^r$ is a $p$-adic non-Liouville number for any $1\leq j\leq m$, and we say that $A$ is $p$-adic non-Liouville if it is $p$-adic non-Liouville in every direction. We say that $A$ has $p$-adic non-Liouville differences in the $r$-th direction if the difference multisubset $A-A:=\{A_i-A_j:1\leq i,j\leq m\}$ is $p$-adic non-Liouville in the $r$-th direction, and we say that $A$ has $p$-adic non-Liouville differences if it has $p$-adic non-Liouville differences in every direction.
\end{defn}

\begin{defn}[D\'efinitions in p.189 of \cite{Gac}, cf.\cite{Ked2} Definition 3.4.3 ]
    \label{equivalence}
For two multisubsets $A=(A_1,\dots,A_m)$ and $B=(B_1,\dots,B_m)$ of $\mathbb{Z}_p^n$, we say that $A$ is weakly equivalent to $B$ if there exists a constant $c>0$ and a sequence of permutations $\sigma_h$ $(h\in \ZZ_{>0})$ of $\{1,2,\dots,m \}$ such that, for all $1\leq i \leq n$ and $1\leq j\leq m$, 
$$ p^h\left< \frac{A^i_{\sigma_h(j)}-B^i_j}{p^h} \right>\leq ch. $$
We say that $A$ is equivalent to $B$ if there exists a permutation $\sigma$ of $\{1,2,\dots,m\}$ such that for all $1\leq i\leq n$ and $1\leq j\leq m$,
$$A_{\sigma(j)}^i-B^i_j\in \ZZ.$$
\end{defn}
These two relations of multisubsets are equivalent relations. Also, it is easy to see that if $A$ and $B$ are equivalent, then they are weakly equivalent. Also, we have the following:
\begin{lem}[cf. \cite{Ked1} Lemma 13.4.3]
For $a,b\in\ZZ_p^n$, if the singleton multisets $\{a\},\{b\}$ are weakly equivalent, they are equivalent.
\label{above6}
\end{lem}
\begin{proof}
    If we write $a=(a^1,\dots,a^n)$, $b=(b^1,\dots,b^n)$, then the singleton multisets $\{a^i\},\{b^i\}$ are weakly equivalent. So they are equivalent by \cite{Ked1} Lemma 13.4.3, and this implies that $\{a\}$ and $\{b\}$ are equivalent.
\end{proof}

\begin{defn}
[cf.  \cite{Ked2}, Definition 3.4.4] Let $A,\AA_1,\dots,\AA_k$ be multisubsets of $\ZZ_p^n$ such that $A=\bigcup_{i=1}^k \AA_i$ as multisets. We say that $\AA_1,\dots,\AA_k$ form a Liouville partition of $A$ in the $r$-th direction if $\AA_1^r,\dots,\AA_k^r$ is a Liouville partition of $A^r$, namely, for any $1\leq l<m\leq k$ and $a_l\in \AA_l^r, a_m\in\AA_m^r$, $a_l-a_m$ is a $p$-adic non-Liouville number which is not an integer.
\end{defn}

\begin{prop}[cf. \cite{Ked2}, Proposition 3.4.5]
Let $A$ be a finite multisubset of $\ZZ_p^n$ and let $\AA_1,\dots,\AA_k$ be a Liouville partition of $A$ in the $r$-th direction.
\begin{enumerate}[(1)]

\item Let $\BB_1,\dots,\BB_k$ be multisubsets of $\ZZ_p^n$ such that $\BB_i^r$ is weakly equivalent to $\AA_j^r$ for $1\leq j\leq k$. Then $\BB_1,\dots,\BB_k$ form a Liouville partition in the $r$-th direction of $B=\bigcup_{j=1}^k \BB_k$.

\item Suppose that $B$ is a multisubset of $\ZZ_p^n$ weakly equivalent to $A$. Then $B$ admits a Liouville partition $\BB_1,\dots,\BB_k$ in the $r$-th direction such that $\BB_j$ is weakly equivalent to $\AA_j$ for $1\leq j\leq k$.

\end{enumerate}
\label{weexp}
\end{prop}

\begin{proof}
(1) By definition, $\AA_1^r,\dots,\AA_k^r$ is a Liouville partition of $A^r$. Then, by \cite{Ked2}, Proposition 3.4.5(a), $\BB_1^r,\dots,\BB_k^r$ is a Liouville partition of $B^r$, and thus $\BB_1,\dots,\BB_k$ is a Liouville partition of $B$ in the $r$-th direction.

(2) Let $A=(A_1,\dots,A_m) $ and $B=(B_1,\dots, B_m)$. Since $A$ is weakly equivalent to $B$, $A^r$ is weakly equivalent to $B^r$. By \cite{Ked2}, Proposition 3.4.5(b), $B^r$ admits a Liouville partition $\BB_1^r,\dots,\BB_k^r$. If we put $\BB_i=:\{ B_g\in B: B_g^r\in \BB_i^r \}$, $\BB_1,\dots,\BB_k$ forms a Liouville partition of $B$ in the $r$-th direction.
\end{proof}

\begin{lem}[cf. \cite{Ked2}, Corollary 3.4.9]\label{above1}
    Let $A,B$ be two finite multisubsets in $\ZZ_p^n$ which are weakly equivalent and suppose that $A$ is $p$-adic non-Liouville. Then $B$ is also $p$-adic non-Liouville.
\end{lem}

\begin{proof}
    By assumption, $A^r$ is $p$-adic non-Liouville for all $1\leq r\leq n$, and $A^r$ and $B^r$ are weakly equivalent. So by \cite{Ked2}, Corollary 3.4.9, $B^r$ is $p$-adic non-Liouville and so is $B$.
\end{proof}

\begin{cor}[cf. \cite{Ked2}, Corollary 3.4.10]
    \label{nexp}
Let $A$ be a finite multisubset of $\ZZ_p^n$ such that $A-A$ is weakly equivalent to a $p$-adic non-Liouville multiset. Then $A$ has $p$-adic non-Liouville differences.
\end{cor}
\begin{proof}
By lemma \ref{above1}, $A-A$ is $p$-adic non-Liouville.
\end{proof}

\begin{defn}
Let $k\geq 1$. We define the notion of Liouville partition of a multisubset $A$ of $\ZZ_p^n$ by $k$ multisubsets $\AA_1,\dots,\AA_k$ of $\ZZ_p^n$ inductively on $k$ as follows:
\begin{enumerate}[(1)]

\item When $k=1$, $\AA_1$ is a Liouville partition of $A$ if $\AA_1=A$ as multisets.

\item For general $k$, $\AA_1,\dots,\AA_k$ is a Liouville partition of $A$ if there exists a partition 
$$ \{1,\dots,k\}=\bigcup_{i=1}^l I_i $$
as sets for some $l\geq 2$ with each $I_i$ nonempty such that $\bigcup_{j\in I_1}\AA_j,\dots,\bigcup_{j\in I_l}\AA_j$ is a Liouville partition in the $r$-th direction of $A$ for some $1\leq r\leq n$ and that $\AA_j\ (j\in I_i)$ is a Liouville partition of $\bigcup_{j\in I_i} \AA_j$, which is defined by induction hypothesis.
\end{enumerate}
\end{defn}

\begin{lem}[cf. \cite{Ked2}, Corollary 3.4.8]
    Let $A,B$ be two finite multisubsets in $\ZZ_p^n$ which are weakly equivalent, and suppose that $A$ has $p$-adic non Liouville differences. Then $A$ and $B$ are equivalent.
    \label{above5}
\end{lem}

\begin{proof}
    Take the partition of $A$ into $\ZZ^n$-cosets. Then it is a Liouville partition because it is obtained as the iteration of the partition into $\ZZ$-cosets in the $r$-th component for $1\leq i\leq n$, which is a Liouville partition in the $r$-th direction since $A^r$ has $p$-adic non-Liouville differences. Then, Proposition \ref{weexp}(2) implies that we have a corresponding Liouville partition of $B$. Then, to prove the Lemma, we are reduced to the case where $A$ is a single $\ZZ^n$-coset, and then to the case where $A$ is a multisubset of $\ZZ^n$. Then, $A$ is equivalent to the multiset consisting only of $(0,\dots,0)\in\ZZ_p^n$ and so we may assume that $A$ itself is of this form. In this case, for any $b\in B$, $\{b\}$ is weakly equivalent to $\{0\}$ and so $b\in \ZZ^n$ by lemma \ref{above6}. So $B$ is a multisubset of $\ZZ^n$ and hence equivalent to $A$.
\end{proof}

\section{Categories $\mathscr{C}_\rho$ and $\mathscr{D}_\rho$}

For a polysegment $I\subset \RR_{>0}^n$, let $\Der(K_{I,n}/K)$ be the module of continuous $K$-derivations on $K_{I,n}$, where the topology on $K_{I,n}$ is induced by $\rho$-Gauss norms for all $\rho\in I$. It is a finite free module generated by derivations with respect to each $t_i$, which are denoted by $\partial_{t_i}$, for $1\leq i\leq n$.

\begin{defn}
Let $P$ be a finite projective $K_{I,n}$-module. A connection over $P$ is a $K$-linear homomorphism $\nabla:\Der(K_{I,n}/K)\to \End_K(P)$ satisfying the Leibniz rule:
$$ \nabla(\partial)(fa)=\partial(f)a+f\nabla(\partial)(a),\ \text{for all}\ \ \partial\in \Der(K_{I,n}/K),\ f\in K_{I,n}, \ a\in P.$$

Moreover, a connection is called integrable if for $\partial$, $\partial'\in \Der(K_{I,n}/K)$, $\nabla([\partial,\partial'])=[\nabla(\partial),\nabla(\partial')]$, where $[\cdot,\cdot]$ is the Lie bracket.
\end{defn}

A (finite projective) differential module over $K_{I,n}$ is a (finite projective) $K_{I,n}$-module $P$ with an integrable connection $\nabla_P$, which we denote simply by $\nabla$ if no ambiguity arises. A horizontal homomorphism between differential modules $P$ and $Q$ over $K_{I,n}$ is a module homomorphism $f:P\to Q$ satisfying $\nabla_Q(\partial)(f(x))=f(\nabla_P(\partial)(x)) $ for all $\partial \in \Der(K_{I,n}/K)$ and for all $x\in P$. In the rest of the paper, we often say a differential module over $K_{I,n}$ a differential module over $I$, by abuse of terminology.

For any $\rho\in \RR_{>0}^n$, there is a direct system of rings $((K_{I,n})_I,(\varphi_{IJ}:K_{I,n}\to K_{J,n})_{I\leq J})$ with the index set being all closed polysegments of $\RR_{>0}^n$ containing $\rho$ in its interior and partially ordered by inverse inclusion, and homomorphisms are given by the canonical inclusion. We denote the direct limit of this direct system by $R_{\rho,n}$ or simply by $R_\rho$. $R_\rho$ is canonically a differential ring with respect to $\partial_{t_i}$, for $1\leq i\leq n$.

Since any finite projective module over $R_\rho$ is extended from a finite projective module over $\Kabn$ for some $\alpha<\rho<\beta$, it is extended from a finite free $\Kabn$-module for some $\alpha<\rho<\beta$ by Theorem \ref{QS}. In particular, all finite projective modules over $R_\rho$ are finite free.

Let $\Der(R_\rho/K)$ be the module of $K$-derivations $\partial$ on $R_\rho $ such that $\partial|_{K_{I,n}}:K_{I,n}\to R_\rho$ is induced by a continuous $K$-derivation on $K_{I,n}$ for any closed polysegment $I$ containing $\rho$ in its interior. It is again a finite free module generated by $\partial_{t_i}$, for $1\leq i\leq n$. Using $\Der(R_\rho/K)$, we can define the notion of (finite free) differential modules over $R_\rho$ and horizontal homomorphisms between them in the same way as above. It is easy to see that any finite free differential module over $R_\rho$ is extended from some finite free differential module over $K_{I,n}$ for some closed interval $I$ containing $\rho$ in its interior.

\begin{defn}[cf. \cite{KX1}, Definition 1.5.2, \cite{Ked1}, Definition 9.4.7, \cite{Ked1}, Definition 13.3.1] Let $I\subset \RR_{>0}^n$ be a polysegment and let $P$ be a finite projective differential module over $K_{I,n}$. Take $\rho\in I$, let $F_\rho$ be the completion of $K(t_1,\dots,t_n)$ with respect to the $\rho$-Gauss norm, and put $V_\rho=P\otimes_{K_{I,n}}F_\rho$. The intrinsic radius of $P$ at $\rho$ is defined as
$$IR(V_\rho)=\min_{1\leq i\leq n}IR_{\partial_{t_i}}(V_\rho)=\min_{1\leq i\leq n}\frac{|\partial_{t_i}|_{\mathrm{sp},F_\rho}}{|\nabla(\partial_{t_i})|_{\mathrm{sp},V_\rho}}\in (0,1].$$
We say that $P$ satisfies the Robba condition if $IR(V_\rho)=1$ for all $\rho\in I$. Also, we say that a finite free differential module over $R_\rho$ satisfies the Robba condition if it is extended from a finite projective differential module over $K_{J,n}$ satisfying the Robba condition for some closed polysegment $J$ containing $\rho$ in its interior.
\label{Robba}
\end{defn}
Let $\mathscr{D}_\rho$ be the category in which objects are finite free differential modules over $R_\rho$ satisfying the Robba condition, and morphisms are horizontal homomorphisms. For a polysegment $J$ containing $\rho$ in its interior, we say that $P$ is an object in $\mathscr{D}_\rho$ defined(by $P'$) over $J$ if $P'$ is a finite free differential module over $K_{J,n}$ and $P$ is extended from $P'$.

For $\zeta=(\zeta_1,\dots,\zeta_n)\in \Gamma^n_s$, an $n$-tuple of variables $t=(t_1,\dots,t_n)$ and an $n$-tuple of $m\times m$ diagonal matrices $$A=(A^1,\dots,A^n)=(\diag(a_{11},\dots,a_{1m}),\dots,\diag(a_{n1},\dots,a_{nm})),$$ we use the following conventions:

\begin{enumerate}[(1)]
 \item $\zeta t:=(\zeta_1 t_1,\dots,\zeta_n t_n)$.

\item $\zeta^A:=\zeta_1^{A^1}\dots\zeta_n^{A^n}$, with $\zeta_i^{A^i}:=\diag(\zeta_i^{ a_{i1}},\dots,\zeta_i^{ a_{im}})$.
\end{enumerate}

There is a natural group action of $\Gamma^n$ on $K_{J,n}\otimes_K K(\Gamma)$: For $\zeta\in\Gamma^n$ and $f\in K_{J,n}\otimes_K K(\Gamma)$, the action is defined as
$$ \zeta^*(f(t))=f(\zeta t). $$
It naturally induces the action of $\Gamma^n$ on $R_\rho\otimes_K K(\Gamma)$.

\begin{defn}
\label{crho}
The category $\cc$ is defined as follows:

The objects are finite free $R_\rho$ modules $P$ endowed with a semilinear group action of $\Gamma^n$ on $P\otimes_KK(\Gamma)$ satisfying the following conditions:
\begin{enumerate}[(1)]

\item $P$ is extended from a finite free $K_{J,n}$-module $P'$ for some closed polysegment $J$  contains $\rho$ in its interior, and the action of $\Gamma^n$ on $P\otimes_KK(\Gamma)$ is induced from some semilinear group action of $\Gamma^n$ on $P'\otimes_KK(\Gamma)$.

\item The action of $\Gamma^n$ is equivariant with respect to the action of $\mathrm{Gal}(K(\Gamma)/K) $ on both $\Gamma^n$ and $P\otimes_K K(\Gamma)$. That is, for $\sigma\in \mathrm{Gal}(K(\Gamma)/K)$, $\zeta \in \Gamma^n$ and $x\in P\otimes_KK(\Gamma)$, we have $$ \sigma(\zeta^*(x))=\sigma(\zeta)^*(\sigma(x)). $$

\item For some basis $e_1,\dots,e_m$ of $P'$ in (1), there exists $l>0$ such that, for each positive integer $k$ and $\zeta \in \Gamma_k^n$, the representation matrix $E(\zeta)$ of $\zeta^*$ with respect to this basis satisfies the inequality $|E(\zeta)|_J\leq p^{lk}$.

\end{enumerate}

The morphisms $f:P\to Q$ of objects in $\mathscr{C}_\rho$ are module homomorphisms satisfying $\zeta^*((f\otimes \mathrm{id})(x))=(f\otimes \mathrm{id})(\zeta^*(x))$ for all $\zeta\in \Gamma^n$ and $x\in P\otimes_K K(\Gamma)$.
\end{defn}
For a polysegment $J$ containing $\rho$ in its interior, we say that $P$ is an object in $\mathscr{C}_\rho$ defined(by $P'$) over $J$ if $P$ admits a $K_{J,n}$-module $P'$ as in Definition \ref{crho}.

For a finite projective differential module $P$ satisfying the Robba condition over $K_{J,n}$, we define the action of $\zeta\in \Gamma^n$ on $P\otimes_KK(\Gamma)$ by $$\zeta^*(x)=\sum_{\alpha\in \ZZ^n_{\geq 0}}(\zeta-1)^\alpha\binom{tD}{\alpha} (x),$$
where $1=(1,\dots,1)$ and $\binom{tD}{\alpha}=\binom{t_1D_1}{\alpha_1}\dots\binom{t_nD_n}{\alpha_n}  $ with $D_i=\nabla(\partial_{t_i})$. We can write the above series also as
$$ \zeta^*(x)=\sum_{\alpha\in \ZZ^n_{\geq 0}}(\zeta-1)^\alpha t^\alpha\frac{D^\alpha }{\alpha !}(x).$$
Note that, to prove the above series converges, it suffices to prove that this series converges in $(F_\rho\otimes_{K_{J,n}} P)\otimes_KK(\Gamma)$ for all $\rho \in J$.
Fix $\rho\in J$. For $\zeta\in \Gamma^n$, we can choose $k\in \ZZ_{>0}$ such that $\zeta\in \Gamma_k^n$. By \cite{Ked1}, Example 2.1.6, we have the estimate $|\zeta_i-1|\leq p^{-\frac{p^{-k+1}}{p-1}}<1$. Choose $\lambda<1$ such that $|\zeta_i-1||\partial|_{\mathrm{sp},F_\rho}<\lambda |\partial|_{\mathrm{sp},F_\rho}$, and then choose $\epsilon$ such that 
$$ |\zeta_i-1|(|\partial_{t_i}|_{\mathrm{sp},F_\rho}+\epsilon) \leq  \lambda |\partial_{t_i}|_{\mathrm{sp},F_\rho}.$$
Then by \cite{Ked1}, Lemma 6.1.8, we can choose $c=c(\epsilon)$ such that 
$$ |D_i^{\alpha_i}|_{F_\rho\otimes P}\leq c(|D_i|_{\mathrm{sp},F_\rho\otimes P}+\epsilon)^{\alpha_i}= c(|\partial_{t_i}|_{\mathrm{sp},F_\rho}+\epsilon)^{\alpha_i}. $$
Then we have 
\begin{align*}
\left|(\zeta-1)^\alpha t^\alpha \frac{D^\alpha}{\alpha !}\right|_{(F_\rho\otimes P)\otimes_K K(\Gamma)} & \leq
| \zeta-1 |^\alpha c^n \prod_{i=1}^n (|\partial_{t_i}|_{\mathrm{sp},F_\rho}+\epsilon)^{\alpha_i} \rho^\alpha \left|\frac{1}{\alpha !}\right| \\
& < c^n \prod_{i=1}^n |\partial_{t_i}|_{\mathrm{sp},F_\rho}^{\alpha_i}\rho^\alpha\left|\frac{1}{\alpha !}\right |\lambda^{|\alpha|}\\
& = c^n \prod_{i=1}^n (\omega \left| \frac{1}{\alpha_i !} \right|^{\frac{1}{\alpha_i}})^{\alpha_i} \lambda^{|\alpha|},
\end{align*}
which tends to $0$ as $|\alpha|\to \infty$, as required. Here we used the fact that $|\partial_{t_i}|_{\mathrm{sp},F_\rho}=\omega\rho_i^{-1}$(\cite{Ked1}, Definition 9.4.1).
Moreover, the action of $\Gamma^n$ is a group action. Indeed, noting that 
\begin{align*}
D^\alpha t^\beta& =\sum_\gamma\binom{\alpha}{\gamma}d^{\alpha-\gamma}(t^\beta)D^\gamma\\
& =\sum_{0\leq \gamma\leq \alpha, \alpha-\beta\leq\gamma\leq \alpha} \frac{\alpha!}{\gamma!(\alpha-\gamma)!}\frac{\beta!}{(\beta-\alpha+\gamma)!}t^{\beta-\alpha+\gamma}D^\gamma,
\end{align*}
we have for $\zeta,\mu\in \Gamma^n$,
\begin{align*}
\zeta^*\mu^*& =\sum_{\alpha\in \ZZ^n_{\geq 0}} (\zeta-1)^\alpha \frac{t^\alpha D^\alpha}{\alpha!}\sum_{\beta\in\ZZ^n_{\geq 0}} (\mu-1)^\beta \frac{t^\beta D^\beta}{\beta!}\\
& = \sum_{\alpha,\beta}(\zeta-1)^\alpha(\mu-1)^\beta \frac{1}{\alpha!\beta!}t^\alpha D^\alpha t^\beta D^\beta\\
& = \sum_{\overset{\alpha,\beta,\gamma}{0\leq \gamma\leq \alpha, \alpha-\beta\leq\gamma\leq \alpha}}(\zeta-1)^\alpha(\mu-1)^\beta\frac{1}{(\beta-\alpha+\gamma)!\gamma!(\alpha-\gamma)!}t^{\beta+\gamma}D^{\beta+\gamma}\\
& \overset{\delta=\beta+\gamma}{=}\sum_{0\leq\gamma\leq\alpha\leq\delta}(\zeta-1)^\alpha(\mu-1)^{\delta-\gamma}\frac{\delta!}{(\delta-\alpha)!\gamma!(\alpha-\gamma)!}\frac{t^\delta D^\delta}{\delta!}\\
& =\sum_{\delta}\left(\sum_{0\leq\gamma\leq\alpha\leq\delta}\left(((\zeta-1)(\mu-1))^{\alpha-\gamma}(\zeta-1)^\gamma(\mu-1)^{\delta-\alpha}\frac{\delta!}{(\delta-\alpha)!\gamma!(\alpha-\gamma)!}\right)\frac{t^\delta D^\delta}{\delta!} \right)\\
& =\sum_{\delta}((\zeta-1)(\mu-1)+(\zeta-1)+(\mu-1))^\delta \frac{t^\delta D^\delta}{\delta!} =(\zeta\mu)^*.
\end{align*}

By the argument above, we can equip an object in $\mathscr{D}_\rho$ with an action of $\Gamma^n$. 
\begin{rmk}
Actually, from the argument above we can also conclude the following. 
\begin{enumerate}[(1)]
    \item For any $\lambda\in K(\Gamma)^n$ such that $|\lambda-1|<1$, the action $\lambda^*=\sum_{\alpha\in\ZZ_{\geq 0}^n}(\lambda-1)^\alpha\binom{tD}{\alpha} (x)$ is always well-defined.
    \item For $\lambda_1,\lambda_2$ such that $|\lambda_i-1|<1$ for $i=1,2$, we always have $\lambda_1^*\lambda_2^*=(\lambda_1\lambda_2)^*$.
\end{enumerate}
\end{rmk}

\begin{lem}
With notations above, the action of $\zeta\in \Gamma^n$ on $K_{J,n}\otimes_K K(\Gamma)$ induced from the canonical differential module structure on $K_{J,n}$ is given by the natural action
$$\zeta^*(f(t))=f(\zeta t).$$
\end{lem}

\begin{proof}

\begin{align*}
 \zeta^*(f(t)) 
=&  \sum_{\alpha\in\ZZ^n_{\geq 0}}(\zeta-1)^\alpha\binom{t_1D_1}{\alpha_1}\dots\binom{t_nD_n}{\alpha_n}\sum_{i\in \ZZ^n} f_it^i \\
=&  \sum_{i\in\ZZ^n}\left(\sum_{\alpha\in\ZZ^n_{\geq 0}} (\zeta-1)^\alpha \binom{i_1}{\alpha_1}\dots\binom{i_n}{\alpha_n}\right)f_it^i \\
= & \sum_{i\in \ZZ^n}\left( \prod_{k=1}^n\left( \sum_{\alpha_k=0}^\infty (\zeta_k-1)^{\alpha_k}\binom{i_k}{\alpha_k}  \right)  \right)f_it^i\\
= & \sum_{i\in \ZZ^n}\left( \prod_{k=1}^n\left((\zeta_k-1)+1\right)^{i_k} \right)f_it^i\\
= & \sum_{i\in \ZZ^n} f_i\zeta^it^i=f(\zeta t).
\end{align*}
\end{proof}

\begin{lem}\label{above2}
An object in $\mathscr{D}_\rho$ endowed with the action of $\zeta\in\Gamma^n$ explained above defines an object in $\mathscr{C}_\rho$.
\end{lem}

\begin{proof}
We check the conditions (1), (2), (3) in Definition \ref{crho}.
(1) We prove that the action of $\Gamma^n$ is a semilinear group action. For $P\in\mathscr{D}_\rho$ over $K_{J,n}$, $x\in P$, $f\in K_{J,n}$ and $\zeta\in \Gamma^n$, we have
\begin{align*}
\zeta^*(fx) = & \sum_{\alpha\in\ZZ^n_{\geq 0}}(\zeta-1)^\alpha t^\alpha \frac{D^\alpha}{\alpha!}(fx)\\
= & \sum_\alpha\sum_{0\leq \beta\leq \alpha} (\zeta-1)^\alpha t^\alpha \frac{1}{\alpha!}\binom{\alpha}{\beta}D^\beta(f) D^{\alpha-\beta}(x)\\
= & \sum_\alpha\sum_{0\leq \beta\leq \alpha} (\zeta-1)^\alpha t^\alpha \frac{D^\beta}{\beta!}(f)\frac{D^{\alpha-\beta}}{(\alpha-\beta)!}(x)\\
\overset{\gamma=\alpha-\beta}{=} & \left( \sum_{\beta}(\zeta-1)^\beta t^\beta \frac{D^\beta}{\beta!} \right)\left( \sum_{\gamma}(\zeta-1)^\gamma t^\gamma \frac{D^\gamma}{\gamma!} \right)=\zeta^*(f)\zeta^*(x).
\end{align*}
Also we have already shown that $(\zeta\eta)^*=\zeta^*\eta^*$. Thus we conclude that this is a semilinear group action.

(2) It suffices to check for elements of the form $x=y\otimes a$ for $y\in P$ and $a\in K(\Gamma)$. In this case, we have
\begin{align*}
\sigma(\zeta^*(y\otimes a))
= & \sigma\left(\sum_{\gamma\in\ZZ^n_{\geq 0}}(\zeta-1)^\gamma\binom{tD}{\gamma}(y\otimes a) \right)\\
= & \left(\sum_{\gamma\in\ZZ^n_{\geq 0}}(\sigma(\zeta)-1)^\gamma\binom{tD}{\gamma}(y\otimes \sigma(a)) \right)
=\sigma(\zeta)^*(\sigma(x)).
\end{align*}

(3) Let $P$ be an object in $\mathscr{D}_\rho$ defined by $P'$ over some closed polysegment $J=\prod_{i=1}^n [\alpha_i,\beta_i]$. Fix a basis $e_1,\dots, e_m$ of $P'$ and define the matrices $E(\zeta)$ for $\zeta\in\Gamma^n$ as in Definition \ref{crho} (3), using this basis. Because $|E(\zeta)|_J$ is equal to the maximaum of $|E(\zeta)|_\sigma$ for 
$$\sigma=(\sigma_1,\dots,\sigma_n)\in\prod_{i=1}^n\{ \alpha_i,\beta_i \},$$
it suffices to prove the required estimate for each $|E(\zeta)|_\sigma$. For $1\leq i\leq n$, let $L_{i,\sigma}$ denote the completion of $K(t_1,\dots,t_{i-1},t_{i+1},\dots,t_n)$ with respect to the $\hat{\sigma}=(\sigma_1,\dots,\sigma_{i-1},\sigma_{i+1},\dots,\sigma_n)$-Gauss norm. Then $P_i:=(L_{i,\sigma})_{[\alpha_i,\beta_i],1}\otimes_{K_{J,n}} P'$ is a finite free differential module of rank $m$ over $(L_{i,\sigma})_{[\alpha_i,\beta_i],1}$ for the derivation $\partial_{t_i}$. Since 
\begin{align*}
    (\zeta\eta)^*(\ee)
    = & \zeta^*\eta^*(\ee)=  \zeta^*((\ee) E(\eta))\\
    = & (\ee)E(\zeta)\zeta^*(E(\eta)),
\end{align*}
for any $\zeta=(\zeta_1,\dots,\zeta_n)\in \Gamma^n$, we can write
$$E(\zeta)=E(\zeta'_1)(\zeta_1')^*(E(\zeta_2'))\cdots(\zeta_1'\cdots\zeta_{i-1}')^*(E(\zeta_n')),$$
where $\zeta'_i=(1,\dots,1,\zeta_i,1,\dots,1)$ with $i$-th entry being $\zeta_i$. Because the action of $\Gamma^n$ on $K_{J,n}$ is an isometry with respect to $|\cdot|_\sigma$, it suffices to prove the required estimate for each $|E(\zeta_i')|_\sigma$, and it is equal to $|E(\zeta_i)|_{\sigma_i}$, where $E(\zeta_i)$ is the representation matrix of the action of $\zeta_i\in\Gamma$ on $P_i$ induced from its differential module structure over $(L_{i,\sigma})_{[\alpha_i,\beta_i],1}$ and $|\cdot|_{\sigma_i}$ is the $\sigma_i$-Gauss norm on $(L_{i,\sigma})_{[\alpha_i,\beta_i],1}$. Thus we can reduce the proof to one dimensional case by putting $K=L_{i,\sigma}$, which is already shown in the proof of \cite{Ked1}, Theorem 13.5.5.
\end{proof}
By Lemma \ref{above2}, we have a functor $\mathscr{D}_\rho\to\mathscr{C}_\rho$.

\begin{rmk}\label{above 4}
\begin{enumerate}[(1)]
    \item For $(P_1,\nabla_1),(P_2,\nabla_2)\in \mathscr{D}_\rho$, we can define the connection $\nabla$ on the tensor product $P_1\otimes P_2$ by $$\nabla(\partial)(x_1\otimes x_2)=\nabla_1(\partial)(x_1)\otimes x_2+x_1\otimes\nabla_2(\partial)(x_2)\quad (\partial\in\Der(R_\rho/K)), $$
    and we can check that $(P_1\otimes P_2,\nabla)$ is an object in $\mathscr{D}_\rho$(The Robba condition follows from Lemma 6.2.8(b) in \cite{Ked1}). On the other hand. for $P_1,P_2\in\mathscr{C}_\rho$, we can define the action of $\Gamma^n$ on the tensor product
    $$ (P_1\otimes P_2)\otimes_K K(\Gamma)=(P_1\otimes_K K(\Gamma))\otimes(P_2\otimes_K K(\Gamma)), $$
    by $\zeta^*(x_1\otimes x_2)=\zeta^*(x_1)\otimes\zeta^*(x_2)$ and we can check that $P_1\otimes P_2\in\mathscr{C}_\rho$.
    Then the functor $\mathscr{D}_\rho\to\mathscr{C}_\rho$ defined above is compatible with tensor product: Indeed, for $P_1,P_2\in\mathscr{D}_\rho$, $x_1\in P_1, x_2\in P_2$ and $\zeta\in\Gamma^n$,
    \begin{align*}
        \zeta^*(x_1\otimes x_2)
        & =\sum_{\alpha}\frac{(\zeta-1)^\alpha}{\alpha !}t^\alpha D^\alpha(x_1\otimes x_2)\\
        & =\sum_{\alpha}\frac{(\zeta-1)^\alpha}{\alpha !}t^\alpha\left( \sum_{\beta+\gamma=\alpha} \frac{\alpha !}{\beta !\gamma !}D^\beta(x_1)\otimes D^\gamma(x_2)\right)\\
        & = \left(\sum_{\beta}\frac{(\zeta-1)^\beta}{\beta !}t^\beta D^\beta(x_1)\right) \otimes \left(\sum_{\gamma}\frac{(\zeta-1)^\gamma}{\gamma !}t^\gamma D^\gamma(x_2)\right)\\
        & =\zeta^*(x_1)\otimes \zeta^*(x_2)
    \end{align*}

    \item For $(P,\nabla)\in\mathscr{D}_\rho$, we can define the connection $\nabla^\vee$ on the dual $P^\vee$ by
    $$ \nabla^\vee(\partial)(\varphi)(x)=\partial(\varphi(x))-\varphi(\nabla(\partial)(x))\quad (\partial\in\Der(R_\rho/K)), $$
    and we can check that $(P^\vee,\nabla^\vee)$ is an object in $\mathscr{D}_\rho$(The Robba condition follows from Lemma 6.2.8(b) in \cite{Ked1}). On the other hand, for $P$ in $\mathscr{C}_\rho$, we can define the action of $\Gamma^n$ on the  dual
    $$ P^\vee\otimes_K K(\Gamma)= \mathrm{Hom}_{K(\Gamma)}(P\otimes_K K(\Gamma), K(\Gamma)) $$
    by $(\zeta^*\varphi)(x)=\zeta^*(\varphi((\zeta^{-1})^*x))$ and we can check that $P^\vee\in\mathscr{C}_\rho$.

    Then the functor $\mathscr{D}_\rho\to\mathscr{C}_\rho$ defined above is compatible with dual: For $P\in\mathscr{D}_\rho,\varphi\in P^\vee,x\in P$ and $\zeta\in\Gamma^n$,
    \begin{align*}
        (\zeta^*\varphi)(x)
        & = \sum_\alpha(\zeta-1)^\alpha\frac{t^\alpha}{\alpha !}(D^\alpha \varphi)(x)\\
        & = \sum_\alpha(\zeta-1)^\alpha\frac{t^\alpha}{\alpha !}\sum_{\beta+\gamma=\alpha}\frac{\alpha!}{\beta !\gamma !}(-1)^\gamma d^\beta(\varphi(D^\gamma(x)))\\
        & = \sum_{\beta,\gamma}(-1)^\gamma(\zeta-1)^{\beta+\gamma}\frac{t^{\beta+\gamma}}{\beta!\gamma!}d^\beta(\varphi(D^\gamma(x))).
    \end{align*}
    On the other hand,
    \begin{align*}
        \zeta^*(\varphi((\zeta^{-1})^*))
        & = \sum_\alpha(\zeta-1)^\alpha \frac{t^\alpha}{\alpha !}d^\alpha\left(\varphi\left(\sum_\beta(\zeta^{-1}-1)^\beta\frac{t^\beta}{\beta!}D^\beta(x)\right)\right)\\
        & = \sum_{\alpha,\beta}\frac{(\zeta-1)^\alpha(\zeta^{-1}-1)^\beta}{\alpha!\beta!}t^\alpha d^\alpha(t^\beta\varphi(D^\beta(x))),
    \end{align*}
    and using the equality
    \begin{align*}
        d^\alpha t^\beta 
        & = \sum_\gamma\binom{\alpha}{\gamma}d^{\alpha-\gamma}(t^\beta)d^\gamma\\
        & = \sum_{0\leq \gamma\leq\alpha \atop \alpha-\beta\leq\gamma\leq\alpha}\frac{\alpha!}{\gamma!(\alpha-\gamma)!}\frac{\beta!}{(\beta-\alpha+\gamma)!}t^{\beta-\alpha+\gamma}d^r,
    \end{align*}
    it is further rewritten as
    \begin{align*}
        & =\sum_{\alpha,\beta} \sum_{0\leq \gamma\leq\alpha \atop \alpha-\beta\leq\gamma\leq\alpha}\frac{(\zeta-1)^\alpha(\zeta^{-1}-1)^\beta}{\alpha!\beta!}t^\alpha \frac{\alpha!}{\gamma!(\alpha-\gamma)!}\frac{\beta!}{(\beta-\alpha+\gamma)!}t^{\beta-\alpha+\gamma}d^r(\varphi(D^\beta(x)))\\
        & = \sum_{\alpha,\beta}\sum_{0\leq \gamma\leq\alpha \atop \alpha-\beta\leq\gamma\leq\alpha} \frac{(\zeta-1)^\alpha(\zeta^{-1}-1)^\beta}{\gamma!(\alpha-\gamma)!(\beta-\alpha+\gamma)!}t^{\beta+\gamma}d^r(\varphi(D^\beta(x)))\\
        & =\sum_{\beta,\gamma} \sum_{\gamma\leq\alpha\leq\beta+\gamma} \frac{(\zeta-1)^\alpha(\zeta^{-1}-1)^\beta}{\gamma!(\alpha-\gamma)!(\beta-\alpha+\gamma)!}t^{\beta+\gamma}d^r(\varphi(D^\beta(x)))\\
        & = \sum_{\beta,\gamma} \sum_{0\leq\alpha\leq \beta}\frac{(\zeta-1)^\gamma(\zeta^{-1}-1)^\beta}{\gamma!}\frac{(\zeta-1)^\alpha}{\alpha!(\beta-\alpha)!}t^{\beta+\gamma}d^r(\varphi(D^\beta(x)))\\
        & = \sum_{\beta,\gamma}\frac{(\zeta-1)^\gamma(\zeta^{-1}-1)^\beta}{\gamma!\beta!}\zeta^\beta t^{\beta+\gamma}d^r(\varphi(D^\beta(x)))\\
        & = \sum_{\beta,\gamma}(-1)^\beta\frac{(\zeta-1)^{\beta+\gamma}}{\beta!\gamma!}t^{\beta+\gamma}d^r(\varphi(D^\beta(x))).
    \end{align*}
    So $(\zeta^*\varphi)(x)=\zeta^*(\varphi((\zeta^{-1})^*x))$, as required.
\end{enumerate}
\end{rmk}

From now on, we will use the symbol $D_i$ to denote $\nabla(t_i\partial_{t_i})$, not $\nabla(\partial_{t_i})$, unless otherwise stated.
To prove that $\mathscr{C}_\rho$ is an abelian category, we prove several lemmas.

\begin{lem}
\label{st}
For a monomial $t^s=t_1^{s_1}\dots t_n^{s_n}$ with $s\neq 0$, there exists $\zeta\in\Gamma^n$ such that $\zeta^*(t^s)=(\zeta t)^s\neq t^s$.
\end{lem}

\begin{proof}
Take an index $i$ with $s_i \not=0$ and take a positive integer $m$ with 
$|s_i| < p^m$. Then, if we define $\zeta = (\zeta_1, \dots, \zeta_n) \in \Gamma^n$ to be the element such that $\zeta_i$ is a primitive $p^m$-th root of unity and that $\zeta_j = 1$ for all $j \not= i$, we see that 
$\zeta^*(t^s)=\zeta_i^{s_i} t^s\neq t^s$.
\end{proof}

\begin{lem}\label{above 3}
Every finitely generated $\Gamma^n$-stable ideal $I$ of $\Kabn\otimes_K K(\Gamma)$ is either zero or the unit ideal.
\end{lem}

\begin{proof}
For $f = \sum_{i \in \ZZ^n} f_i t^i\in \Kabn$ with $f_0 \not= 0$, define the finite set 
$S(f)$ by 
$$ S(f) := \{i \in \ZZ^n-\{0\} \,:\, |f_it^i|_{[\alpha,\beta]} \geq |f_0|\}. $$
Suppose $I \not= 0$ and take an element $f \in I$ with nonzero constant term 
such that the cardinality of $S(f)$ is minimal, which is possible because 
$t_1,\dots,t_n$ are invertible in $\Kabn \otimes_K K(\Gamma)$. It suffices to prove that 
$S(f)$ is empty, because it implies that $f$ is invertible and so $I$ is the unit ideal.

Suppose that $S(f)$ is not empty and take $s \in S(f)$. Then there exists 
an element $\zeta$ in $\Gamma^n$ with $\zeta^*(t^s) = \zeta^st^s \not= t^s$ by Lemma \ref{st}. 
Define $g = \sum_{i \in \ZZ^n}g_it^i$ by 
$$ g := p^{-m}\sum_{k=1}^{p^m} (\zeta^k)^*(f) = 
\sum_{i \in \ZZ^n} \left( p^{-m} \sum_{k=1}^{p^m} \zeta^{ik} f_i \right) 
t^i, $$
where $p^m$ is the order of $\zeta \in \Gamma^n$. 
Then $g_i$ is equal to $f_i$ if $\zeta^i=1$ and $0$ otherwise. 
In particular, we have $g_0 = f_0, g_s = 0$ and $|g_i| \leq |f_i|$ in general. 
So the constant term of $g$ is nonzero and 
$S(g)$ is strictly contained in $S(f)$. On the other hand, $g$ belongs to $I$ because $I$ is a $\Gamma^n$-stable ideal. This contradicts the definition of $f$ and so $S(f)$ is empty, as required.
\end{proof}

\begin{lem}
\label{proj}
Let $Q$ be a finite module over $\Kabn$ with a semilinear group action of $\Gamma^n$ on $Q\otimes_K K(\Gamma)$. Then $Q$ is projective.
\end{lem}

\begin{proof}
Let $F=\mathrm{Frac} \Kabn$. Then there exsits an isomorphism $Q\otimes_\Kabn F\to F^r$ for some $r$. 
Then we can choose $g\neq 0$ such that the above isomorphism comes from an isomorphism $Q\otimes_\Kabn \Kabn[\frac{1}{g}]\to (\Kabn[\frac{1}{g}])^r  $. Then $Q$ is free on $\{x\in \Spec(\Kabn):g(x)\neq 0\}\subset \Spec(\Kabn)$, and $Q\otimes_K K(\Gamma)$ is hence locally free on $\bigcup_{\zeta\in \Gamma^n}\{\zeta^*(g)\neq 0\}\subset \Spec(\Kabn\otimes_K K(\Gamma))$. The complement of this set is defined by a nonzero ideal $I=\{\zeta(g):\zeta\in \Gamma^n\}$, which is a nonzero $\Gamma^n$-stable ideal of $\Kabn\otimes_K K(\Gamma)$, and so it is the unit ideal by Lemma \ref{above 3}. So $Q\otimes_K K(\Gamma)$ and hence $Q$ is projective.
\end{proof}

\begin{prop}
$\mathscr{C}_\rho$ is an Abelian category.
\end{prop}
\begin{proof}
Let $f: P\to Q$ be a morphism in $\mathscr{C}_\rho$. Then it comes from some homomorphism $f':P'\to Q'$, where $P'$ (resp. $Q'$) is a finite projective module over $K_{J,n}$ for some closed polysegment $J$ containing $\rho$ in its interior with a semilinear group action of $\Gamma^n$ on $P'\otimes_KK(\Gamma)$ (resp. $Q'\otimes_KK(\Gamma)$), and $f'\otimes \mathrm{id}: P'\otimes_K K(\Gamma)\to Q'\otimes_K K(\Gamma)$ is compatible with $\Gamma^n$-action. Both kernel and cokernel of $f'\otimes \mathrm{id}:P'\otimes_K K(\Gamma)\to Q'\otimes_K K(\Gamma)$(as module homomorphism) admit an action of $\Gamma^n$ and are projective by Lemma \ref{proj}. Note that the kernel and cokernel of $f$ are induced from that of $f'$, thus both kernel and cokernel of $f$ are free. A monomorphism (resp. epimorphism) $f:P\to Q$ is the kernel (resp. cokernel) of $Q\to Q/P$ (resp. $\ker(f)\to P$). 
\label{cab}
\end{proof}

\begin{rmk}
    It is easy to see that the functor $\mathscr{D}_\rho\to \mathscr{C}_\rho$ is exact. So, by Remark \ref{above 4}, it is an exact tensor functor of abelian categories.
\end{rmk}

\section{Generalized $p$-adic Fuchs theorem}
\subsection{$p$-adic exponents for objects in $\mathscr{C}_\rho$}

Now we define $p$-adic exponents for objects in $\mathscr{C}_\rho$ and prove basic properties of them.
\begin{defn}[cf. \cite{Ked1}, Definition 13.5.1  and \cite{Ked2}, Definition 3.4.11]
\label{exp}
    Let $P$ be an object in $\mathscr{C}_\rho$ free of rank $m$. Take $\alpha<\rho<\beta$ such that $P$ is defined by $P'$ over $[\alpha,\beta]$, and take a basis $\ee$ of $P'$. An exponent of $P$(admitted by $P'$) is an $n$-tuple of $m\times m$ diagonal matrices of $A=(A^1,\dots,A^n)$ with entries in $\ZP$ for which there exists a sequence $\{S_{k,A}\}_{k=1}^\infty$ in $\mathrm{M}_n(\Kabn)$ satisfying the following conditions.
\begin{enumerate}[(1)]
\item  If we put 
$(v_{k,A,1},\dots,v_{k,A,m})=(e_1,\dots,e_m)S_{k,A},$
then for all $\zeta\in \Gamma_k^n$
$$\zeta^*(v_{k,A,1},\dots,v_{k,A,m})=(v_{k,A,1},\dots,v_{k,A,m})\zeta^A.$$

\item There exists $l>0$ such that $|S_{k,A}|_{[\alpha,\beta]}\leq p^{lk}$ for all $k$.

\item We have $|\det(S_{k,A})|_{[\alpha,\beta]}\geq 1$ for all $k$.
\end{enumerate}

For an object $P$ in $\mathscr{D}_\rho$, we say $A$ is an exponent of $P$, if  when considered as an object in $\mathscr{C}_\rho$, $A$ is an exponent of $P$.
\end{defn}

An exponent $A$ defined above can be regarded as a multisubset $A$ of $\ZZ_p^n$ by identifying $A_j$ with $(A^1_j,\dots,A^n_j)$(where $A_j^i$ is the $(j,j)$-th entry of $A^i$) and putting $A=(A_1,\dots,A_n)$. So we will use these two notations interchangeably (for the convenience of using conclusions in Section 1) as long as no ambiguity arises.  Moreover, once we fix $P'$, the definition of exponent does not depend on the choice of basis, and to define the exponent, it is enough to find matrices $S_{k,A}$ satisfying (1), (2) and (3) for sufficitently large $k$, because the matrix $S_{k,A}$ satisfies the conditions required for the matrices $S_{k',A}$ for any $k'<k$.

By \cite{Gac}, page $188$, Lemma 4, for matrices $S_{k,A}$ defined above, there exists $k_0>0$ such that $S_{k,A}$ is invertible for $k>k_0$. Then, from the description of the inverse matrix using cofactors, we have $|S_{k,A}^{-1}|_{[\alpha,\beta]}\leq p^{(m-1)kl}$. 

Also, for the given basis, the representation matrix of action of $\zeta\in \Gamma^n$ and $S_{k,A}$ has the following relation:
$$ S_{k,A}(\zeta t)=E(\zeta)^{-1}S_{k,A}(t) \zeta^A .$$
Here $E(\zeta)$ is the representation matrix of the action of $\zeta$ with respect to the basis $\ee$ fixed above.

\begin{thm} [cf. \cite{Ked1}, Theorem 13.5.5,\cite{Gac}, Th\'eor\`eme in p.173] 
    Let $P$ be an object in $\mathscr{C}_\rho$. Then there exists an exponent $A$ for $P$.
\label{exexp}
\end{thm}

\begin{proof}
Suppose that $P$ is defined by $P'$ over a closed polysegment $J$ containing $\rho$ in its interior, and fix a basis $\ee$ of $P'$. First, for any multisubset $A$ of $\ZZ_p^n$ with cardinality $m$, we wish to define the matrix $S_{k,A}$ satisfying (1) and (2) in Definition \ref{exp}. To achieve this, we set 
$$S_{k,A}=p^{-nk}\sum_{\zeta\in \Gamma^n_k}E(\zeta)\zeta^{-A}.$$
The property (2) in Definition \ref{crho} of $\mathscr{C}_\rho$ implies that, for $\sigma\in \mathrm{Gal}(K(\Gamma)/K)$ and $\zeta\in\Gamma^n$, $E(\sigma(\zeta))=\sigma(E(\zeta))$. Thus we have $$ \sigma(S_{k,A})=p^{-nk}\sum_{\zeta \in \Gamma^n_k}E(\sigma(\zeta))\sigma(\zeta)^A=S_{k,A}, $$
and this implies that $S_{k,A}$ has entries in $\Kabn$. Also, if we put \break$ (\vv{k}{A})=(\ee )S_{k,A}$, we have 

\begin{align*}
\zeta^*(v_{k,A,j})
= & \sum_{i=1}^m\zeta^*((S_{k,A})_{ij}e_i)=
  \sum_{i=1}^m(p^{-nk}\sum_{\eta\in \Gamma^n_k}\zeta^*(E(\eta)_{ij}\eta_1^{-A^1_j}\dots \eta_n^{-A^n_j}e_i)\\
= & \sum_{i=1}^m(p^{-nk}\sum_{\eta\in \Gamma^n_k}(\zeta^*E(\eta))_{ij}\eta_1^{-A^1_j}\dots \eta_n^{-A^n_j}\sum_{r=1}^m E(\zeta)_{ri}e_r\\
= & p^{-nk}\sum_{\eta\in \Gamma^n_k}\sum_{r=1}^m(E(\zeta\eta)\eta^{-A})_{rj}e_r.
\end{align*}
Thus we have 
\begin{align*}
\zeta^*(\vv{k}{A})
= & (\ee) p^{-nk}\sum_{\zeta\eta\in \Gamma^n_k}E(\zeta\eta)(\zeta\eta)^{-A}\zeta^A\\
= & (\ee) S_{k,A}\zeta^A=(\vv{k}{A})\zeta^A,
\end{align*}
as desired.
(2) is satisfied because by assumption $|E(\zeta)|_J\leq p^{lk}$ for $\zeta \in \Gamma_k^n$, and we have 
$$ |S_{k,A}|_J\leq |p^{-nk}|\max_{\zeta\in\Gamma^n_k}|E(\zeta)|_J\leq p^{(n+l)k}. $$
Then we choose $A$ so that (3) is satisfied. Note that, for $\zeta\in\Gamma_{k+1}$, $\sum_{i=0}^{p-1}\zeta^{-p^k i}$ equals to $p$ if $\zeta\in\Gamma_k$, and equals to $0$ otherwise. By a simple computation, we have
\begin{align*}
&\ \ \ \ \sum_{b_1,\dots,b_n=0}^{p-1}S_{k+1,A+p^k(b_1I,\dots,b_n I)}\\
& = p^{-n(k+1)}\sum_{\zeta\in \Gamma_{k+1}^n} \sum_{b_1,\dots,b_n=0}^{p-1}E(\zeta)\zeta^{-A-p^k(b_1I,\dots,b_nI)}\\
& = p^{-n(k+1)}\sum_{\zeta\in \Gamma_{k+1}^n} E(\zeta)\zeta^{-A}\prod_{i=1}^n\left(\sum_{b_i=0}^{p-1}\zeta_i^{-p^kb_iI}\right)\\
& = p^{-nk}\sum_{\zeta\in \Gamma^n_k} E(\zeta)\zeta^{-A}=S_{k,A}.
\end{align*}
Now if we put $S_{k,A}=(S_{k,A,1},\dots,S_{k,A,m})$, we have 
$$ S_{k,A,i}=\sum_{b_1,\dots,b_n}^{p-1} S_{k+1,A+p^k(b_1I,\dots,b_nI),i}. $$
So
\begin{align*}
\det(S_{k,A})& =\sum_{i=1}^m\sum_{b_{1i},\dots,b_{ni}=0}^{p-1}\det(S_{k+1,A+p^k(b_{1i}I,\dots,b_{ni}I),i})_{1\leq i\leq m}\\
& = \sum_{\overset{b_{ij}=0}{1\leq i,j\leq m}}^{p-1}\det(S_{k+1,A+p^k(\diag(b_{11},\dots,b_{1m}),\dots,\diag(b_{n1},\dots,b_{nm}))}).
\end{align*}
Denote the constant term of $\det(S_{k,A})$ by $\det(S_{k,A})_0$. Then, for any multisubset $A$ of $\ZZ_p^n$, there exists $B\equiv A\mod p^k$ such that $|\det(S_{k+1,B})_0|\geq|\det(S_{k,A})_0|$ by the above equality. Since $S_{0,A}=I$, we can choose a suitable $A$ by induction on $k$ such that $|\det(S_{k,A})_0|\geq 1$ for any $k$. Then, since $|\det(S_{k,A})|_\rho\geq |\det(S_{k,A})_0|$ for any $\rho\in[\alpha,\beta]$, we obtain the assertion.
\end{proof}

\begin{thm}[cf. \cite{Ked1}, Theorem 13.5.6] 
Let $P$ be an object in $\mathscr{C}_\rho$ defined by $P_1$ over $J_1$ and by $P_2$ over $J_2$, where $J_1,J_2$ are closed polysegments containing $\rho$ in its interior. Then the exponents of $P$ defined by $P_1$ and $P_2$ are weakly equivalent. In particular, the exponent of $P$ is uniquely determined up to weak equivalence.
\label{weakeq}
\end{thm}
\begin{proof}
First, we consider the situation where $J_1=J_2:=J$, and $P_1\cong P_2:=P_J$.
Let $A$ and $B$ be two exponents of $P$ defined by $P_J$ and its basis $\ee$, and let $S_{k,A}$ and $S_{k,B}$ be the corresponding sequences of matrices defining each exponent. Set 
$(\vv{k}{A})=(\ee) S_{k,A}$ and $(\vv{k}{B})=(\ee) S_{k,B}$. For a sufficiently large $k_0$, let $T_k=S^{-1}_{k,A}S_{k,B}$ for $k>k_0$ be the change-of-basis matrix from $\vv{k}{A}$ to $\vv{k}{B}$. Then we have 
$$ |T_k|_J\leq |S^{-1}_{k,A}|_J|S_{k,B}|_J\leq p^{kml} .$$
The determinant of $T_k$ is bounded from below as follows:
$$ |\det(T_k)|_\rho\geq|\det(S_{k,A})|_\rho^{-1}\geq |S_{k,A}|^{-m}_\rho\geq p^{-kml},\ \ \ \rho\in J .$$
 Write $J=[\alpha,\beta]$, and put $\gamma_i=\sqrt{\alpha_i\beta_i} $ for $1\leq i\leq n$. Then, by the definition of determinant, there must be a permutation $\sigma_k$ of $\{ 1,2,\dots,m\}$ such that 
$$ \prod_{i=1}^m|(T_k)_{i,\sigma_k(i)}|_\gamma\geq p^{-kml}. $$
Then, for a single factor appearing in the product above, we have 
$$  |(T_k)_{i,\sigma_k(i)}|_\gamma\geq p^{-kml}\prod_{j\neq i} |(T_k)_{j,\sigma_k(j)}|_\gamma^{-1}\geq p^{-m^2kl}. $$
We write $(T_k)_{i,\sigma_k(i)}=\sum_{j\in \ZZ^n}T_{kij}t^j$ with $T_{kij}\in K$. We can then choose $j=j(i,k)=(j_1,\dots,j_n)\in\ZZ^n$ such that $|T_{kij}t^j|_\gamma=|(T_k)_{i,\sigma_k(i)}|_{\gamma}$. Put $\eta_r=\sqrt{\beta_r/\alpha_r} $.  Then, for $1\leq r\leq n$, in the case $j_r\geq 0$ we have 
$$ \eta_r^{j_r}=\left( \frac{\beta_r}{\gamma_r} \right)^{j_r}= \frac{|T_{ijk}t^j|_{\beta'_r}}{|T_{ijk}t^j|_{\gamma}}\leq p^{(m^2+m)lk}, $$
and in the case $j_r<0$, we have
$$ \eta_r^{-j_r}=\left( \frac{\alpha_r}{\gamma_r} \right)^{j_r}= \frac{|T_{ijk}t^j|_{\alpha'_r}}{|T_{ijk}t^j|_{\gamma}}\leq p^{(m^2+m)lk}, $$
where $\beta'_r=(\gamma_1,\dots,\gamma_{r-1},\beta_r,\gamma_{r+1},\dots,\gamma_n)$, and $\alpha'_r=(\gamma_1,\dots,\gamma_{r-1},\alpha_r,\gamma_{r+1},\dots,\gamma_n)$.
In either case we deduce that 
$$ |j_r|\log_p\eta_r\leq {(m^2+m)lk}. $$
Finally, for $\zeta\in\Gamma_k^n$ and sufficiently large $k$, we have
\begin{align*}
\zeta^*(T_k) & =\zeta^*(S_{k,A}^{-1})\zeta^*(S_{k,B})\\
& = \zeta^{-A}S_{k,A}^{-1}E(\zeta)E(\zeta)^{-1}S_{k,B}\zeta^B=\zeta^{-A}T_k\zeta^B.
\end{align*}
By looking at the coefficients of $t^j$ of the $(i,\sigma_{k}(i))$-th entries of the matrices in the above equation, we obtain the equality
$$ \zeta^jT_{kij}=\zeta^{-A_i+B_{\sigma_k(i)}}T_{kij} .$$
Since this equality holds for any $\zeta\in \Gamma_k^n$, we see that 
$$ A_i-B_{\sigma_k(i)}-j(i,k)\in p^k\ZZ_p^n .$$
Thus we conclude that $A$ and $B$ are weakly equivalent.

Now we consider the case where $J_1$ and $J_2$ are not necessarily the same. Then we can find $J\subset J_i$ of positive length containing $\rho$ in its interior for $i=1,2$ such that $(P_1)_J$ and $(P_2)_J$ are isomorphic. By definition, an exponent of $P$ defined by $P_i$ is an exponent of $P$ defined by $(P_i)_J$. Hence we can reduce to the previous case and so we are done.
\end{proof}

\begin{lem}[cf. \cite{Ked2}, Remark 3.4.14, \cite{KS}]
Let $P_1$, $P_2$ and $P$ be three objects in $\mathscr{C}_\rho$  with $P_i$ having exponent $A_i\ (i=1,2)$. Then,
\begin{enumerate}[(1)]
\item if there exists a short exact sequence
$$0\to P_1\to P\to P_2\to 0,$$
then $P$ admits the multiset union $A_1\cup A_2$ as an exponent.

\item the module $P_1\otimes P_2$ is an object in $\mathscr{C}_\rho$, and admits the multiset $A_1+A_2$ as an exponent.

\item the module $P_1^\vee $ is an object in $\mathscr{C}_\rho$, and admits the multiset $-A_1$ as an exponent.
\end{enumerate}
\label{exp123}
\end{lem}

\begin{proof}
We may assume that $P_1,P_2,P$ is defined by $P_1',P_2',P'$ respectively over $J$ for some closed polysegment $J$ containing $\rho$ in its interior such that $A_i$ is an exponent of $P_i$ defined by $P'_i$ for $i=1,2$. 

(1) By shrinking $J$, we may suppose moreover that the exact sequence in the statement is induced by an exact sequence
$$ 0\to P'_1\to P'\to P_2'\to 0 $$
which is compatible with $\Gamma^n$-action when tensored with $K(\Gamma)$ over $K$.
Let $\ee$ be a basis of $P'_{1}$ and choose $e'_1,\dots,e'_{m'}\in P'$ lifting a basis of $P'_{2}$. Then, for $\zeta\in\Gamma^n_k$, the representation matrix of $\zeta^*$ on basis $(e_1,\dots,e_m,e'_1,\dots,e'_{m'})$ is a block matrix $\widetilde{E(\zeta)}=$
$
\begin{pmatrix*}
E(\zeta) & *\\
O & E(\zeta)'\\
\end{pmatrix*}
$
with $E(\zeta)$ and $E(\zeta)'$ the representation matrix of the action of $\zeta$ on $P'_{1}$ and $P'_{2}$ respectively, with respect to the chosen basis. This matrix has norm at most $p^{l'k}$ for some $l'$ by condition (3) of Definition \ref{crho}. Then, for each $k>0$, there exist $\vv{k}{A_1}\in P'_{1}$ and $ v_{k,A_2,1},\dots,v_{k,A_2,m'} \in \langle e'_1,\dots,e'_{m'} \rangle\subset P'$  such that for any $\zeta\in \Gamma_k^n$, the following are satisfied:
\begin{align*}
(\ee )S_{k,A_1} =\ & (\vv{k}{A_1}),\\
(e'_1,\dots,e'_{m'})S_{k,A_2} =\ & (v_{k,A_2,1},\dots,v_{k,A_2,m'}),\\
\zeta^*(\vv{k}{A_1})=\ & (\vv{k}{A_1})\zeta^{A_1}, \\
\zeta^*(v_{k,A_2,1},\dots,v_{k,A_2,m'})\equiv\  & (v_{k,A_2,1},\dots,v_{k,A_2,m'})\zeta^{A_2}\ ( \mathrm{mod}\ P'_{1}),\\
|S_{k,A_i}|_J\leq p^{l_ik}, |S^{-1}_{k,A_i}|_J & \leq p^{(m+m'-1)l_ik}, i=1,2.\\
\end{align*}
Now we define 
$$ v'_{k,A_2,j}= p^{-nk}\sum_{\zeta\in\Gamma^n_k}\zeta_1^{-(A_2)^1_j}\dots\zeta_n^{-(A_2)^n_j}\zeta^*(v_{k,A_2,j}). $$
Then we observe that 
\begin{align*}
\zeta^*(v'_{k,A_2,j})
& =  p^{-nk}\sum_{\eta\in\Gamma^n_k}\eta_1^{-(A_2)^1_j}\dots\eta_n^{-(A_2)^n_j}(\zeta\eta)^*(v_{m,A_2,j})\\
& = p^{-nk}\sum_{\zeta\eta\in\Gamma^n_k}\zeta_1^{(A_2)^1_j}\dots\zeta_n^{(A_2)^n_j}(\zeta\eta)_1^{-(A_2)^1_j}\dots(\zeta\eta)_n^{-(A_2)^n_j}(\zeta\eta)^*(v_{m,A_2,j})\\
& = \zeta_1^{(A_2)^1_j}\dots\zeta_n^{(A_2)^n_j}v'_{k,A_2,j}.\\
\end{align*}
Also, the matrix $S_{k,A}$ sending $\ee,e'_1,\dots,e'_{m'}$ to $\vv{k}{A_1},v'_{k,A_2,1},\break \dots,v'_{k,A_2,m'}\in P'_{1}$ is a block matrix 
$
\begin{pmatrix*}
S_{k,A_1}& *\\
O & S_{k,A_2}\\
\end{pmatrix*}
$
. To prove that it has the properties required to define an exponent, it suffices to prove that the norm of the * part of the matrix is bounded by $p^{ks}$ for some $s$. First we write $v_{k,A_2,j}$ as 
$$ v_{k,A_2,j}=\sum_{i=1}^{m'} (S_{k,A_2})_{ij}e'_i. $$
Then we have 
\begin{align*}
\zeta^*(v_{k,A_2,j})& =\sum_{i=1}^{m'}\zeta^*(S_{k,A_2})_{ij}\zeta^*(e'_i)\\
& = \sum_{i=1}^{m'}\zeta^*(S_{k,A_2})_{ij}\sum_{l=1}^{m+m'}(\widetilde{{E(\zeta)}})_{li}e_l,
\end{align*}
where we put $e_{m+i}:=e'_i$ for $1\leq i\leq m'$. Thus
\begin{align*}
v'_{k,A_2,j}& =p^{-nk}\sum_{\zeta\in\Gamma^n_k}\zeta_1^{-(A_2)^1_j}\dots\zeta_n^{-(A_2)^n_j}\zeta^*(v_{m,A_2,j})\\
& = p^{-nk}\sum_{\zeta\in\Gamma^n_k}\zeta_1^{-(A_2)^1_j}\dots\zeta_n^{-(A_2)^n_j}\left( \sum_{i=1}^{m'}\zeta^*(S_{k,A_2})_{ij}\sum_{l=1}^{m+m'}(\widetilde{{E(\zeta)}})_{li}e_l \right)\\
& = \sum_{l=1}^{m+m'}\left( p^{-nk}\sum_{\zeta\in \Gamma_k^n}\sum_{i=1}^{m'}\zeta^*(S_{k,A_2})_{ij} \widetilde{E(\zeta)}_{li} \right)e_l.
\end{align*}
Note that the * part of the matrix consists of the coefficients of $\ee$ in the above equality, and its norm is hence bounded by $p^{ks}$ if we choose $s=n+l_1+l_2$.
So we see that $A_1\cup A_2$ is an exponent of $P$.

(2) Let $e_1,\dots,e_m$ be a basis of $P'_{1}$, $e'_1,\dots,e'_{m'}$ be a basis of $P'_{2}$, and $S_{k,A_1},\ S_{k,A_2}$ be matrices defining the exponents, with $|S_{k,A_1}|_J\leq p^{ks_1}$ and $|S_{k,A_2}|_J\leq p^{ks_2}$. It is easy to see that $e_i\otimes e'_j$ is a basis of $P'_1\otimes P'_2$, for $1\leq i\leq m,\ 1\leq j\leq m'$ and the norm of $S_{k,A_1}\otimes S_{k,A_2} $ is bounded by $p^{k(s_1+s_2)}$. Here the tensor product of matrices is the Kronecker tensor product.
For $\zeta\in\Gamma^n_k$, we have
$$ \zeta^*(v_{k,A_1,i}\otimes v'_{k,A_2,j})_{1\leq i\leq m, 1\leq j \leq m'}=(v_{k,A_1,i}\otimes v'_{k,A_2,j})_{1\leq i\leq m, 1\leq j \leq m'} \zeta^{A_1}\otimes\zeta^{A_2}, $$
and by a quick computation of $\zeta^{A_1}\otimes\zeta^{A_2}$ we see that $A_1+A_2$ is an exponent of $P_1\otimes P_2$.

(3) Let $\ee$ be a basis of $P_1'$, and suppose that $\vv{k}{A_1}$, $S_{k,A_1}$ define the exponent. For sufficiently large $k$ such that $S_{k,A_1}$ is invertible, $\zeta\in \Gamma^n_k$, and $v=l_1v_{k,A_1,1}+\dots+l_mv_{k,A_1,m}\in P_1'$,
\begin{align*}
\zeta^*(v^*_{k,A_1,j})(v)
= & \sum_{i=1}^m l_i\zeta^*( v^*_{k,A_1,j}((\zeta^{-1})^*(v_{k,A_1,i})))\\
= & \zeta_1^{-(A_1)^1_{j}}\dots\zeta_n^{-(A_1)^n_{j}}v^*_{k,A_1,j}(v),\\
\end{align*}
where $v^*_{k,A_1,1},\dots, v^*_{k,A_1,m}$ is the dual basis of $\vv{k}{A_1}$.
The change-of-basis matrix from $e^*_1,\dots,e^*_m$ to $v^*_{k,A_1,1},\dots, v^*_{k,A_1,m}$ is $(S_{k,A_1}^T)^{-1}$, which satisfies desired properties.
\end{proof}

\begin{defn}[cf.\cite{Ked2}, Definition 3.4.18]
Let $P$ be an object in $\mathscr{C}_\rho$. We say that $P$ has $p$-adic non-Liouville exponents if $P$ admits an exponent $A$ over $J$ which is $p$-adic non-Liouville. This is equivalent to say that every exponent of $P$ is $p$-adic non-Liouville by Lemma \ref{above1}. We say that $P$ has $p$-adic non-Liouville exponent differences if $\End(P) $ has $p$-adic non-Liouville exponents.
\end{defn}

Lemma \ref{exp123} implies that, by viewing $\End(P)$ as $P\otimes P^\vee$, if $P$ admits an exponent with non-Liouville difference, then $P$ has $p$-adic non-Liouville exponent differences.

\begin{cor}[cf. \cite{Ked2}, Lemma 3.4.19]
The following statements are equivalent:
\begin{enumerate}[(1)]
\item The module $P$ has $p$-adic non-Liouville exponent differences.

\item Some exponent of $P$ has $p$-adic non-Liouville differences.

\item Every exponent of $P$ has $p$-adic non-Liouville differences.
\end{enumerate}
\label{expdiff}
\end{cor}

\begin{proof}
By the argument above (2) implies (1). (3) obviously implies (2).

Suppose now that (1) holds. Let $A$ be any exponent for $P$, and let $B$ be an exponent for $\End(P)$ which is $p$-adic non-Liouville. Then $A-A$ is an exponent for $\End(P)$, and weakly equivalent to $B$ by Theorem \ref{weakeq}. So $A-A$ is also $p$-adic non-Liouville by Corollary \ref{nexp}, implying (3). 
\end{proof}

\subsection{Generalized $p$-adic Fuchs theorem for objects in $\mathscr{C}_\rho$ and $\mathscr{D}_\rho$}

\begin{lem}[cf. \cite{Ked2}, Lemma 3.1.5]
Let $\eta>1 $, and let $\alpha,\ \alpha',\ \beta,\ \beta'\in \RR^n_{>0}$ such that $\alpha'<\beta'$, $\alpha'=\alpha\eta$ and $\beta=\beta'\eta$. Let $f=\sum_{i\in \ZZ}f_i t_r^i$ be an element of $\Kabn$(with each $f_i$ a formal Laurent series in $t_1,\dots,t_{r-1},t_{r+1},\dots,t_n$) such that $f_i=0$ for $|i|<N$. Then 
$$ |f|_{[\alpha',\beta']}\leq \eta^{-N}|f|_{[\alpha,\beta]} $$

\label{simcom}
\end{lem}
\begin{proof}
Put $\alpha=(\alpha_1,\dots,\alpha_n),\alpha'=(\alpha'_1,\dots,\alpha'_n),\beta=(\beta_1,\dots,\beta_n),$ $\beta'=(\beta_1',\dots,\beta_n')$. Then $\alpha'_i=\alpha_i\eta$ and $\beta'_i=\beta_i\eta^{-1}$.
For a monomial $t^s=t_1^{s_1}\cdots t_n^{s_n}$ with $|s_i|\geq N$, let $\gamma_i$ be $\alpha_i$ if $s_i<0$ and $\beta_i$ if $s_i\geq 0$, and let $\gamma'_i$ be $\alpha'_i$ if $s_i<0$ and $\beta'_i$ if $s_i\geq 0$. Then $r_i^{s_i}=(\gamma'_i)^{s_i}\eta^{|s_i|}$. Hence
$$ |t^s|_{[\alpha,\beta]}=\prod_{i=1}^n \gamma_i^{s_i}=\eta^{|s|}\prod_{i=1}^n(\gamma_i')^{s_i}\geq \eta^N|t^s|_{[\alpha',\beta']}, $$
as required.

In general case, we can take $s\in\ZZ^n$ such that $|f|_{[\alpha',\beta']}=|f_st_r^s|_{[\alpha',\beta']}$. Then, by the case of monomials,
$$ |f|_{[\alpha',\beta']}= |f_st_r^s|_{[\alpha',\beta']}\leq \eta^{-N}|f_st_r^s|_{[\alpha,\beta]}\leq \eta^{-N}|f|_{[\alpha,\beta]}. $$
\end{proof}

\begin{lem}[cf. \cite{Ked2}, Theorem 3.4.22 and \cite{KS}]
Let $P$ be an object in $\mathscr{C}_\rho$ admitting an exponent $A$ with nontrivial Liouville partition $A=B\cup C$ in the $r$-th direction. Then $P$ admits a nontrivial decomposition of the form $P=P_1\oplus P_2$.
\label{dec}
\end{lem}
\begin{proof}
Suppose that $P$ is defined by $P'$ over a closed polysegment $J=[\alpha,\beta]$ with $\alpha<\rho<\beta$ which admits $A$ as an exponent.
Choose $\eta>1$ and $\alpha',\ \beta'$ with $\frac{\alpha'_i}{\alpha_i}=\frac{\beta_i}{\beta'_i}=\eta$ such that $\alpha'<\rho<\beta'$.

Choose a basis $e_1,\dots,e_m$ of $P'$ and let $S_{k,A}(k\geq 1)$ be matrices defining the exponent $A$ which are invertible as elements in $\mathrm{M}_m(K_{[\alpha,\beta],n})$. Choose $\lambda\in(0,1)$ and a sufficiently large constant $d$ such that
$$ p^{5ml}\eta^{-d}\leq \lambda .$$

Since $B\cup C$ form a Liouville partition in the $r$-th direction of $A$, we can reorder $A$ and write \begin{equation*}
A=\left(
\begin{pmatrix}
B_1 &  \\
  &C_1
\end{pmatrix}
,\dots,
\begin{pmatrix}
    B_n &  \\
      &C_n
\end{pmatrix}
\right),
\end{equation*}
where $B=(B_1,\dots,B_n)$ and $C=(C_1,\dots,C_n)$. For $d$ chosen above, there is $k_0>0$ such that $h\equiv b_r-c_r\ (\mathrm{mod}\ p^k)$ forces $|h|\geq dk$ for all $k>k_0$ and $b_r\in B_r$, $c_r\in C_r$.

Let $(v_{k,A,1},\dots,v_{k,A,m})=(e_1,\dots,e_m)S_{k,A},$ and let $\Pi_k$ be the projector onto the submodule of $M_{[\alpha,\beta]}$ generated by $v_{k,A,i}$ for those $i$ for which $A_i\in B$, and we denote the representation matrix of $\Pi_k$ with respect to basis $(e_1,\dots,e_m)$ by $N_k$. Then we have 
\begin{align*}
\Pi_k(\ee)
=\ & \Pi_k(\vv{k}{A})S^{-1}_{k,A}\\
=\ & (\vv{k}{A})
\begin{pmatrix*}
 I & O\\
 O & O\\
\end{pmatrix*}
S_{k,A}^{-1}\\
=\ & (\ee) S_{k,A} 
\begin{pmatrix*}
    I & O\\
    O & O\\
\end{pmatrix*}
S_{k,A}^{-1}.
\end{align*}
Thus $|N_k|_{[\alpha,\beta]}\leq p^{mkl}$. Also notice that

\begin{align*}
(\Pi_k-\Pi_{k+1})(\vv{k}{A})
 =\ & (\vv{k}{A})
\begin{pmatrix*}
    I & O\\
    O & O\\
\end{pmatrix*}\\
 - & \Pi_{k+1}(\vv{k+1}{A})S^{-1}_{k+1,A}S_{k,A}\\
 =\ & (\vv{k+1}{A})S^{-1}_{k+1,A}S_{k,A}
\begin{pmatrix*}
    I & O\\
    O & O\\
\end{pmatrix*}\\
- & (\vv{k+1}{A})
\begin{pmatrix*}
    I & O\\
    O & O\\
\end{pmatrix*}
 S^{-1}_{k+1,A} S_{k,A}\\
 =\ & (\vv{k+1}{A}) L_k \\
\end{align*}
with 
\begin{align*}
L_k=S^{-1}_{k+1,A} S_{k,A}
\begin{pmatrix*}
    I & O\\
    O & O\\
\end{pmatrix*}
-
\begin{pmatrix*}
    I & O\\
    O & O\\
\end{pmatrix*}
S^{-1}_{k+1,A} S_{k,A}.
\end{align*}
Then we have $|L_k|_{[\alpha,\beta]}\leq p^{(k+1)l(m-1)+kl}\leq p^{2mlk}$.
 
Moreover, if we write 
\begin{align*}
S^{-1}_{k+1,A} S_{k,A}=E_k=
\begin{pmatrix*}
E_{11}&E_{12}\\
E_{21}&E_{22}\\
\end{pmatrix*}, 
\end{align*}
we can compute the action of $\zeta$ on $L_k$ as follows:
\begin{align*}
\zeta^*(L_k)
= & \zeta^*\left(S^{-1}_{k+1,A} S_{k,A}
\begin{pmatrix*}
    I & O\\
    O & O\\
\end{pmatrix*}
-
\begin{pmatrix*}
    I & O\\
    O & O\\
\end{pmatrix*}
S^{-1}_{k+1,A} S_{k,A}\right)\\
= & \zeta^{-A}S^{-1}_{k+1,A} S_{k,A}\zeta^A
\begin{pmatrix*}
    I & O\\
    O & O\\
\end{pmatrix*}
-
\begin{pmatrix*}
    I & O\\
    O & O\\
\end{pmatrix*}
\zeta^{-A}S^{-1}_{k+1,A} S_{k,A}\zeta^A\\
= & 
\begin{pmatrix*}
O& -\zeta^{-B}E_{12}\zeta^C\\
-\zeta^{-C}E_{21}\zeta^B& O \\
\end{pmatrix*}.
\end{align*}
This means that for $(i,j)$-th entry in $E_{21}$,
\begin{align*}
\zeta^*(E_{21})_{ij}
& = \zeta_1^{-c_{1i}}\dots\zeta_n^{-c_{ni}}(E_{21})_{ij}\zeta_1^{b_{1j}}\dots\zeta_n^{b_{nj}}\\
& = \zeta_1^{b_{1j}-c_{1i}}\dots\zeta_n^{b_{nj}-c_{ni}}(E_{21})_{ij},
\end{align*}
Where $c_{ki}$(resp.$b_{kj}$) denotes the $(i,i)$-th(resp.$(j,j)$-th) entry of $C_k$(resp.$B_k$). 
Thus the coefficient of $t^h$ in $(E_{12})_{ij}$ can only be non-zero if $h_l\equiv b_{lj}-c_{li}\ (\mathrm{mod}\ p^k)$, and this condition for $l=r$ implies that $|h_r|\geq dk$. This along with Lemma \ref{simcom} implies that, for $k\geq k_0$,
$$|(E_{21})_{ij}|_{[\alpha',\beta']}\leq p^{2mlk}\eta^{-dk}. $$ 
By a similar argument we obtain the same result for $E_{12}$. Thus we have $|L_k|_{[\alpha',\beta']}\leq p^{2mlk}\eta^{-dk} $.
Note that 
$$ (\Pi_k-\Pi_{k+1})(\ee)=(\ee) S_{k+1,A}L_k S_{k,A}^{-1} ,$$
we obtain the inequality $|N_k-N_{k+1}|_{[\alpha',\beta']}\leq p^{4mlk}\leq \lambda^k$. So $\Pi_k$ converges to a projector of $P'_{[\alpha',\beta']}$ which is compatible with the action of $\Gamma^n$, and gives the desired decomposition of $P'_{[\alpha',\beta']}$ and hence the desired decomposition of $P$.
\end{proof}

\begin{lem}[cf. \cite{KS}]
    \label{decom_in_crho}
Let $P$ be an object in $\mathscr{C}_\rho$ with an exponent $A$ admitting a Liouville partition $A=\AA_1\cup\dots\cup\AA_k$ in the $r$-th direction. Then $P$ admits a unique decomposition $P_1\oplus\dots\oplus P_k$ in $\mathscr{C}_\rho$ for which $P_i$ admits an exponent weakly equivalent to $\AA_i$ for $1\leq i\leq k$.
\end{lem}

\begin{proof}
We may reduce to the case $k=2$.
We firstly prove that the existence of such a decomposition guarantees the uniqueness. Suppose that $P$ decomposes as $P'_1\oplus P'_2$ and $P''_1\oplus P''_2$ with $P'_i,\ P''_i$ having exponents $\AA'_i,\ \AA''_i$ weakly equivalent to $\AA_i$. Then, by Proposition \ref{weexp}, $\AA'_1$ and $\AA''_2$ forms a Liouville partition of $\AA'_1\cup \AA''_2$ in the $r$-th direction while $\AA'_2$ and $\AA''_1$ forms a Liouville partition of $\AA'_2\cup \AA''_1$ in the $r$-th direction.  Let $f$ be the composition of the inclusion $P'_1\to P$ with the projection $P\to P''_2$. Since $\mathscr{C}_\rho$ is an Abelian category, we have exact sequences:
$$0\to\ker (f)\to P'_1\to \Im(f)\to 0,$$
$$0\to \Im(f)\to P''_2\to \coker(f)\to 0.$$
Choose exponents $C_1, C_2, C_3$ for $\ker(f)$,$\Im(f)$ and $\coker(f)$, respectively. Then, by Lemma \ref{exp123}, $\AA'_1$ is weakly equivalent to $C_1\cup C_2$ and $\AA''_2$ is weakly equivalent to $C_2\cup C_3$. Since $\AA'_1\cup \AA''_2$ is a Liouville partition in the $r$-the direction, $(C_1\cup C_2)\cup(C_2\cup C_3)$ is a Liouville partition in the $r$-th direction. Thus $C_2=\emptyset$, and  $f=0$. If we define $g$ to be the composition of the inclusion $P'_2\to P$ with the projection $P\to P''_1$, by a similar argument we may conclude that $g=0$. So we have $P'_1=P''_1$ and $P'_2=P''_2$.

Now we verify the existence of such a decomposition by induction on the rank of $P$. Without loss of generality, we may assume that both $\AA_1$ and $\AA_2$ are non-empty. Then by Lemma \ref{dec}, there exists a non-trivial decomposition $P=P_1'\oplus P_2'$ in $\mathscr{C}_\rho$. Choose exponents $\AA'_1$ and $\AA'_2$ for $P'_1$ and $P'_2$, respectively. Then by Lemma \ref{exp123} $\AA'_1\cup \AA'_2$ is also an exponent for $P$, and hence weakly equivalent to $\AA_1\cup \AA_2$. By Proposition \ref{weexp}, $\AA'_1\cup \AA'_2$ admits a Liouville partition $\AA''_1\cup \AA''_2$ in the $r$-th direction with $\AA_i''$ weakly equivalent to $\AA_i$, for $i=1,2$. 

Define $\AA''_{ij}$ for $i,j=1,2$ as follows:
$$\AA''_{ij}:=\{ x=(x_1,\dots,x_n)\in \AA'_i: x_r\in (\AA''_j)^r \}.$$
Since $\AA''_1\cup \AA''_2$ is a Liouville partition of $\AA'_1\cup \AA'_2$ in the $r$-th direction, $(\AA''_1)^r\cup (\AA''_2)^r$ is a Liouville partition of  $(\AA'_1)^r\cup (\AA'_2)^r$. Thus $((\AA''_1)^r\cap (\AA'_i)^r)\cup((\AA''_2)^r\cap (\AA'_i)^r)$ is a Liouville partition of $(\AA'_i)^r$, from which we conclude that $\AA''_{i1}\cup \AA''_{i2}$ is a Liouville partition of $\AA'_i$ in the $r$-th direction for $i=1,2$. On the other hand, since
$$ (\AA''_{1j})^r\cup (\AA''_{2j})^r= (\AA''_j)^r, $$
we conclude that $\AA''_{1j}\cup \AA''_{2j}=\AA''_j$ for $j=1,2$.

By induction hypothesis we may decompose $P'_i$ as $P'''_{i1}\oplus P'''_{i2}$ with $P'''_{ij}$ admitting an exponent $\AA'''_{ij}$ weakly equivalent to $\AA''_{ij}$. Now we define $P_i=P'''_{1i}\oplus P'''_{2i}$ so that $P \simeq P_1\oplus P_2$, and by construction $P_i$ admits the exponent $\AA'''_{1i}\cup \AA'''_{2i}$ which is weakly equivalent to $\AA''_{1i}\cup \AA''_{2i}=\AA''_i$, hence weakly equivalent to $\AA_i$ as desired.
\end{proof}

\begin{thm}[cf.\cite{KS}]\label{decomposition_of_D_mod}
Let $P$ be an object in $\mathscr{D}_\rho$ having an exponent $A$ with Liouville partition $\AA_1,\dots,\AA_k$ in the $r$-th direction. Then there exists a unique direct sum decomposition $P=P_1\oplus\dots\oplus P_k$ in $\mathscr{D}_\rho$ with each $P_i$ having exponent weakly equivalent to $\mathscr{A}_i$ for $1\leq i\leq k$.
\end{thm}

\begin{proof}
We may reduce to the case $k=2$. Let $\AA_1\cup \AA_2$ be a Liouville partition in the $r$-th direction. Then by Lemma \ref{decom_in_crho}, $P$ admits a  unique decomposition $P_1\oplus P_2$ in $\mathscr{C}_\rho$ with exponent of $P_i$ weakly equivalent to $\AA_i$. For $\lambda\in (K^\times)^n$ with $|1-\lambda_i|\leq 1$ and for $1\leq i\leq n$, the substitution $t\mapsto \zeta t$ and $t\mapsto \lambda t$ commute. 
Consequently, the pullback $P_{i,\lambda}$ of $P_i$ along $\lambda^*$ are again objects in $\mathscr{C}_\rho$ with the same exponent as that of $P_i$ for $i=1,2$. So, by the uniqueness of decomposition in Lemma \ref{decom_in_crho}, $P_1\oplus P_2$ and $P_{1,\lambda}\oplus P_{2,\lambda}$ must coincide, which implies that $\lambda^*$ preserves $P_1$
and $P_2$. Now for $1\leq i\leq n$, take a sequence of $\lambda=(1,\dots,1,\lambda_i,1,\dots,1)$ with $\lambda_i$ converges to $1$, and consider the action of $\frac{\lambda^*-1}{\lambda-1}$ on $P_1$ and $P_2$. As $\lambda_i$ goes to $1$, this operator converges to $D_i$, from which we conclude that this decomposition in $\mathscr{C}_\rho$ is also a decomposition in $\mathscr{D}_\rho$.
\end{proof}

\begin{rmk}\label{above4}
    Let the notations be as in Theorem \ref{decomposition_of_D_mod}. Then, by definition of the category $\mathscr{D}_\rho$, we have the following: If $P$ is defined by $P'$ over an open polysegment $I$ containing $\rho$, we have a decomposition of $P'=P'_1\oplus\cdots\oplus P'_k$ of $P'$ by finite free differential modules inducing the decomposition $P=P_1\oplus\cdots\oplus P_k$ of the theorem, after shrinking $I$ suitably.

    Moreover, such a decomposition is unique: Indeed, we can check the uniqueness after restricing to closed subpolysegments $J$ of $I$ containing $\rho$ in its interior because $P'_i=\varprojlim_J P_{i,J}',P_i''=\varprojlim_J P''_{i,J}$. Then the cokernels of the inclusions $P_i'\hookrightarrow P'_i+P_i'',P''_i\hookrightarrow P_i'+P_i''$ are finite projective by Lemma \ref{proj}, and they are zero after tensoring with $R_\rho$ by the unique assertion of Theorem \ref{decomposition_of_D_mod}. Thus the cokernels are zero and so $P'_i=P''_i$, as required.
\end{rmk}

\subsection{Generalized $p$-adic Fuchs theorem over polyannuli}
Now we consider a finite projective differential module $P$ over an open polysegment $I$ of $\RR_{>0}^n$ satisfying the Robba condition. Then, for $\rho\in I$, $P_\rho:=P\otimes_{K_{I,n}}R_\rho$ is an object of $\mathscr{D}_\rho$ and so an exponent $A_\rho$ of $P_\rho$ is defined.

Then we have the following independence result:

\begin{lem}
\label{glue_up}
Let the notations be as above. Then, for any $\rho,\rho'\in I$, $A_\rho$ and $A_{\rho'}$ are weakly equivalent.
\end{lem}

\begin{proof}
For any $\rho$, there exists a closed subpolysegment $I_\rho$ of $I$ containing $\rho$ in its interior such that $P_{I_\rho}$ admits an exponent $\wt{A_\rho}$ weakly equivalent to $A_\rho$, by definition of $A_\rho$ and Theorem \ref{weakeq}. Then it suffices to show that $\wt{A_\rho}$ and $\wt{A_{\rho'}}$ are weakly equivalent. If $\mathrm{int}I_\rho\cap \mathrm{int}I_{\rho'}\neq \emptyset$, $\wt{A_\rho},\wt{A_{\rho'}}$ are both weakly equivalent to $A_\sigma$ for $\sigma\in \mathrm{int}I_\rho\cap \mathrm{int}I_{\rho'}$ by Threorem \ref{weakeq} and so they are weakly equivalent. In general case, Since $I$ is connected, we can find a sequence $\rho=\rho_0,\rho_1,\dots,\rho_l=\rho'$ such that $\mathrm{int}I_{\rho_{i-1}}\cap \mathrm{int}I_{\rho_i}\neq \emptyset$ for $1\leq i\leq l$. So we can reduce to the previous case.
\end{proof}

\begin{defn}
Let $P$ be a finite projective differential module over an open polysegment $I$ in $\RR^n_{\geq 0}$ satisfying the Robba condition. We say that $A$ is an exponent of $P$ if it is an exponent of $P_\rho:=P\otimes_{K_{I,n}} R_\rho$ for some $\rho\in I$.
\end{defn}

By Lemma \ref{glue_up}, the above definition of exponent of $P$ does not depend on the choice of $\rho\in I$ up to weak equivalence.

\begin{thm}
\label{decom_int}
Let $P$ be a finite projective differential module over some open polysegment $I$ satisfying the Robba condition, admitting an exponent $A$ with Liouville partition $\AA_1,\dots,\AA_k$ in the $r$-th direction. Then there exists a unique decomposition $P=P_{1}\oplus\dots\oplus P_{k}$ where each $P_{i}$ is a finite projective differential module admitting an exponent weakly equivalent to $\AA_i$ for $1\leq i\leq k$.
\end{thm}

\begin{proof}
For $\rho\in I$, put $P_\rho:=P\otimes_{K_{I,n}} R_\rho$ and let $A_\rho$ be an exponent of $P_\rho$. Then $A_\rho$ is weakly equivalent to $A$ and so it admits a Liouville partition $\AA_{\rho,1},\dots,\AA_{\rho,k}$ in the $r$-th direction such that $A_{\rho,i}$ is weakly equivalent to $\AA_i$ for $1\leq i\leq k$. By Theorem \ref{decomposition_of_D_mod} and Remark \ref{above4}, there exists an open subpolysegment $I_\rho$ of $I$ containing $\rho$ such that $P_{I_\rho}$ is free and admits the unique decomposition
$$ P_{I_\rho}=P_{I_\rho,1}\oplus\cdots\oplus P_{I_\rho, k} $$
by finite free differential modules over $I_\rho$ such that each $P_{I_\rho,i}$ admits an exponent weakly equivalent to $\AA_{\rho,i}$, thus weakly equivalent to $\AA_i$. By the uniqueness assertion in Remark \ref{above4}, we have the equality
$$ P_{I_\rho,i}|_{I_\rho\cap I_{\rho'}}= P_{I_{\rho'},i}|_{I_\rho\cap I_{\rho'}}$$
for $\rho,\rho'\in I$.
Then since we have the exact sequence
$$0\to P\to \prod_{\rho\in I}P_{I_\rho}\rightrightarrows\prod_{\rho,\rho'\in I}P_{I_\rho\cap I_{\rho'}},$$
if we define
$$ P_i:=\ker\left( \prod_{\rho\in I}P_{I_\rho,i}\rightrightarrows\prod_{\rho,\rho'\in I}P_{I_\rho,i}|_{I_\rho\cap I_{\rho'}}\right), $$
we obtain the desired decomposition $P=P_1\oplus\cdots\oplus P_k$. The uniqueness follows from the argument in Remark \ref{above4}.
\end{proof}

By Theorem \ref{decom_int}, we decompose inductively in various directions, and obtain the following:

\begin{cor}\label{decom_all}
Let $P$ be a finite projective differential module over some open polysegment $I$ satisfying the Robba condition, admitting an exponent $A$ with Liouville partition $\AA_1,\dots,\AA_k$. Then there exists a unique decomposition $P=P_{1}\oplus\dots\oplus P_{k}$, where each $P_{i}$ is a finite projective differential module and admits an exponent weakly equivalent to $\AA_i$ for $1\leq i\leq k$.
\end{cor}

\begin{eg}
Temporary in this example, let $D_i=\nabla(\partial_{t_i})$. For each $\lambda\in K^n$, the differential module $M_\lambda$ over a polysegment $I$ in $\RR^n_{>0}$ generated by a single element $v\neq 0$ with $D_i (v)=\lambda_i t_i^{-1}v$ satisfies the Robba condition if and only if $\lambda\in \ZZ_p^n$. Moreover, if $\lambda\in \ZZ_p^n$, $\{\lambda\}$ is an exponent of $M_\lambda$.
\end{eg}
\begin{proof}
By simple computation, we have $D_i^s(v)=s!\binom{\lambda_i}{s}t_i^{-s} v$. Thus the representation matrix of $D_i^s$(which is an element of $K_{I,n}$ because we are considering the rank $1$ case) is $s!\binom{\lambda_i}{s}t_i^{-s} $. So, for $\rho\in I$, if we define $F_\rho$ as in Definition \ref{Robba},  
\begin{align*}
|D_i|_{\mathrm{sp},M_\lambda\otimes F_\rho} & = \max\left(|\partial_{t_i}|_{\mathrm{sp},F_\rho},\limsup_{s\to \infty}\left|s!\binom{\lambda_i}{s}t_i^{-s}\right|_\rho^{\frac{1}{s}}\right)\\
& = \max\left( \omega\rho^{-1},\omega\rho^{-1}\limsup_{s\to\infty}\left|\binom{\lambda_i}{s}\right|^{\frac{1}{s}} \right)
\end{align*}
by \cite{Ked1}, Lemma 6.2.5 and Definition 9.4.1. Thus $M_\lambda$ satisfies the Robba condition if and only if $\limsup_{s\to\infty}|\binom{\lambda_i}{s}|^{\frac{1}{s}}\leq 1$ for all $1\leq i\leq n$. This is equivalent to the claim that the series $\sum_{s=0}^\infty\binom{\lambda_i}{s}X^s$ converges on the $p$-adic open unit disc for all $1\leq i\leq n$, and by \cite{ItoG} Chapter IV Proposition 7.3, this condition is satisfied if and only if $\lambda_i\in\ZZ_p$ for all $1\leq i\leq n$.

When $\lambda\in \ZZ_p^n$, the action of $\zeta\in \Gamma^n$ on the basis $v$ of $M_\lambda\otimes_KK(\Gamma)$ is given by
\begin{align*}
\zeta^*(v) & = \sum_{\alpha\in\ZZ_{\geq 0}^n} (\zeta-1)^\alpha  \binom{tD}{\alpha}(v)\\
& = \sum_{\alpha\in\ZZ_{\geq 0}^n}(\zeta-1)^\alpha \binom{\lambda}{\alpha}(v)=\zeta^\lambda v. 
\end{align*}
Thus $\{\lambda\}$ is an exponent of $M_\lambda$.
\end{proof}

\begin{lem}
\label{hellokitty}
Let $A\in \GL_m(\Kabn)$ such that $|I_m-A|_{[\alpha,\beta]}\leq \lambda<1$. Then $|I_m-A^{-1}|_{[\alpha,\beta]}\leq \lambda$. In particular, $|A|_{[\alpha,\beta]}=|A^{-1}|_{[\alpha,\beta]}=1$.
\end{lem}
\begin{proof}
We can write $A=I_m-(I_m-A)$. Since $|I_m-A|_{[\alpha,\beta]}\leq \lambda<1$, we have $A^{-1}=\sum_{i=0}^\infty (I_m-A)^i$. Thus $|I_m-A^{-1}|_{[\alpha,\beta]}=|\sum_{i=1}^\infty (I_m-A)^i|_{[\alpha,\beta]}\leq \lambda$.
\end{proof}

The following proposition is a special case of \cite{Gac}, Lemme 1 on page 206, but we give a proof based on the argument in \cite{Ked2} for the completeness of the paper.

\begin{prop}
Let $P$ be a free differential module over an open polysegment $I$ in $ \RR^n_{>0}$ satisfying the Robba condition, and having an exponent identically equal to some $\lambda=(\lambda_1,\dots,\lambda_n)\in\ZZ_p^n$. Then for any closed subpolysegment $J$, $M_J$ admits a basis on which $D_i$ acts via a matrix over $K$ whose eigenvalues are all equal to $\lambda_i$ for each $1\leq i\leq n$.
\end{prop}

\begin{proof}
Since $P=(P\otimes M_\lambda^\vee)\otimes M_\lambda$, and $P\otimes M_\lambda^\vee$ has exponent identically equal to $0$ by Lemma \ref{exp123}, we may reduce the theorem to the case $\lambda=0$.
For any closed subpolysegment $J$ of $I$, choose $\eta>1$ and $\alpha,\beta,\alpha',\beta'\in I$ with $\frac{\alpha'_i}{\alpha_i}=\frac{\beta_i}{\beta'_i}=\eta$ for $1\leq i\leq n$, and such that $J\subset [\alpha',\beta']$. Note that, since $P_J$ admits an exponent weakly equivalent to $0$, which is in fact equivalent to $0$ by Lemma \ref{above5}, $P_J$ admits $0$ as an exponent.

Let $S_{k,A}$ be the matrices defining the exponent $A=0$. Choose $\tau\in (0,1)$ and sufficiently large $c>0$ such that 
$$p^{8m^2l}\eta^{-c}\leq \tau.$$
Then for $c$ chosen above, there is a positive integer $k_0$ such that $p^k>ck$ and that $S_{k,A}$ belongs to $\GL_m(\Kabn)$ for all $k\geq k_0$. We will construct matrices $R_k$ in $\GL_m(K)$ for $k\geq k_0$ such that $R_{k_0}=I_m$ and 
\begin{align}
| I_m-R_kS_{k_,A}^{-1}S_{k+1,A}R_{k+1}^{-1} |_{[\alpha',\beta']}\leq\tau^k\ \ (k\geq k_0),    
\end{align}
by induction.
First, note that this assumption implies the following:
\begin{align}
&\ \ \ \  |S_{k_0,A}^{-1}S_{k,A}R_k^{-1}-S_{k_0,A}^{-1}S_{k+1,A}R_{k+1}^{-1}|_{[\alpha',\beta']}\notag\\
& =|S_{k_0,A}^{-1}S_{k,A}R_k^{-1}(I_m-R_kS_{k_,A}^{-1}S_{k+1,A}R_{k+1}^{-1})|_{[\alpha',\beta']}\leq \tau^k|S_{k_0,A}^{-1}S_{k,A}R_k^{-1}|_{[\alpha',\beta']}.
\end{align}
Also (1) implies that $|R_kS_{k_,A}^{-1}S_{k+1,A}R_{k+1}^{-1}|_{[\alpha',\beta']}\leq 1$ by Lemma \ref{hellokitty}, and we have
\begin{align*}
|S_{k_0,A}^{-1}S_{k,A}R_k^{-1}|_{[\alpha',\beta']}& =|(R_{k_0}S_{k_0,A}^{-1}S_{k_0+1,A}R_{k_0+1}^{-1})(R_{k_0+1}S_{k_0+1,A}^{-1}S_{k_0+2,A}R_{k_0+2}^{-1})\\
& \dots(R_{k-1}S_{k-1,A}^{-1}S_{k,A}R_k^{-1})|_{[\alpha',\beta']}\leq 1.
\end{align*}
Thus $ |S^{-1}_{k_0,A}S_{k,A}R_k^{-1}-S^{-1}_{k_0,A}S_{k+1,A}R_{k+1}^{-1}|_{[\alpha',\beta']} \leq \tau^k$, and 
\begin{align*}
& |I_m-S_{k_0,A}^{-1}S_{k,A}R_k^{-1}|_{[\alpha',\beta']}\leq \max\{ |I_m-S_{k_0,A}^{-1}S_{k_0,A}R_{k_0}^{-1}|_{[\alpha',\beta']},|S_{k_0,A}^{-1}S_{k_0,A}R_{k_0}^{-1}\\
& -S_{k_0,A}^{-1}S_{k_0+1,A}R_{k_0+1}^{-1}|_{[\alpha',\beta']},\dots,| S_{k_0,A}^{-1}S_{k-1,A}R_{k-1}^{-1}-S_{k_0,A}^{-1}S_{k,A}R_{k}^{-1} |_{[\alpha',\beta']} \} <1.
\end{align*}

Now we begin to construct $R_{k+1}$. First, we prove that the existence of the sequence up to $R_k$ implies that  
$$|R_k|,\ |R_k^{-1}|\leq p^{mlk}.$$
This is obvious for $k=k_0$. Now assume that $k>k_0$. Then the inequality $|I_m-S_{k_0,A}^{-1}S_{k,A}R_k^{-1}|_{[\alpha',\beta']}<1$ implies that 
$$ |S_{k_0,A}^{-1}S_{k,A}R_k^{-1}|_{[\alpha',\beta']}=| R_k S_{k,A}^{-1}S_{k_0,A} |_{[\alpha',\beta']}=1, $$
by Lemma \ref{hellokitty}. Thus 
$$ |R_k|=|R_k S_{k,A}^{-1}S_{k_0,A}S_{k_0,A}^{-1}S_{k,A}|_{[\alpha',\beta']}\leq p^{k_0l(m-1)+kl}\leq p^{mlk} ,$$
and 
$$ |R_k^{-1}|=|S_{k,A}^{-1}S_{k_0,A}S_{k_0,A}^{-1}S_{k,A} R_k^{-1}|_{[\alpha',\beta']}\leq p^{k_0l+(m-1)lk}\leq p^{mlk},$$
as required. 
Then we define $R_k$ inductively: Given $R_{k_0},\dots,R_k$, put $T_k=R_kS_{k,A}^{-1}S_{k+1,A}$. Then we have 
$$ |T_k|_{[\alpha,\beta]} \leq p^{4m^2kl},\ |T_k^{-1}|_{[\alpha,\beta]}\leq p^{4m^2kl}.$$
Let $T_{k,0}$ be the constant coefficient of $T_k$. For $\zeta\in \Gamma_k^n$, we have 
$$ \zeta^*(T_k)=\zeta^*(R_kS_{k,A}^{-1}S_{k+1,A})=R_k\zeta^{-A}S_{k,A}^{-1}S_{k+1,A}\zeta^A= T_k. $$
This implies that entries of $T_k$ are series in $t^{p^k}$, and by Lemma \ref{simcom}, $|T_k-T_{k,0}|_{[\alpha',\beta']}\leq p^{4m^2kl}\eta^{-p^k} .$
Now we take $R_{k+1}=T_{k,0}$. Because
\begin{align*}
|I_m-R_{k+1}T_k^{-1}|_{[\alpha',\beta']} & \leq |T_k^{-1}|_{[\alpha',\beta']}|T_k-T_{k,0}|_{[\alpha',\beta']}\\
& \leq p^{8m^2kl}\eta^{-p^k}\leq p^{8m^2kl}\eta^{-ck} \leq \tau^k <1,
\end{align*}
$ |I_m-T_kR_{k+1}^{-1}|_{[\alpha',\beta']}\leq \tau^k $. This completes the construction of $R_{k+1}$.
Note that (2) implies that the sequence $S_{k_0,A}^{-1}S_{k,A}R_k^{-1}$ converges to a matrix $U\in \GL_m(K_{[\alpha',\beta'],n})$. Now we prove that $S_{k_0,A}U$ is the change-of-basis matrix to a basis $\ee$ of $P_{[\alpha',\beta']}$ of the desired form. To do so, let $N^r$ denote the matrix of action of $D_r$ on $\ee$. Writting $N^r_{ij}=\sum_{s\in \ZZ^n} N^r_{ijs}t^s$, we have 
\begin{align*}
\sum_{i=1}^m(\sum_{s\in\ZZ^n}N^r_{ijs}t^s)e_i&=D_r( e_j)\\
& = (D_r\circ\zeta)(e_j)=(\zeta\circ D_r)(e_j)\\
& = \sum_{i=1}^m(\sum_{s\in \ZZ^n}\zeta^s N^r_{ijs}t^s) e_i
\end{align*}
for any $\zeta\in\Gamma^n$. Thus $N^r$ is a constant matrix. Let $b$ be an eigenvalue of $N^r$. To use a result in one dimensional case, set $L_i$ to be the completion of $K(t_1,\dots,t_{i-1},t_{i+1},\dots,t_n)$ with respect to the $\hat{\rho}=(\rho_1,\dots,\rho_{i-1},\rho_{i+1},\dots,\rho_n)$-Gauss norm and denote the $i$-th factor of $J$ by $J_i$. Then $P_i:=(L_i)_{J_i,1}\otimes_{K_{J,n}} P$ is a differential module of rank $m$ over $(L_i)_{J_i,1}$ for the derivation $t_i\partial_{t_i}$.
Then by \cite{Ked1}, Proposition 13.3.2(which is a proposition in one dimensional case) for $P_i$,  we have $b\in \ZZ_p$. Let us go back to the original situation and let $v=c_1e_1+\dots+c_me_m\ (c_i\in K)$ be a corresponding eigenvector. Let $\zeta\in\Gamma^n_k$ be such that all components of $\zeta$ are $0$ except the $r$-th component $\zeta_r$. Then we have 
$$ \zeta_r^b v=\zeta^*(v)=\zeta^*(c_1e_1+\dots+c_me_m)=c_1e_1+\dots+c_me_m=v. $$
Thus $b=0$, which complete the proof.
\end{proof}

\begin{prop}\label{freeness}
Let $P$ be a finite projective differential module of rank $m$ over an open polysegment $I$ in $\RR_{>0}^n$ satisfying the Robba condition with exponent indentically equal to $0$. Then  $P$ is free.
\end{prop}

\begin{proof}
For a subpolysegment $J\subset I$, we define $V_J\subset P_J$ by $V_J:=\bigcap_{j=1}^n \ker(D_j^m)$.

For $\rho\in I$, there exists an open subpolysegment $I_\rho\subset I$ containing $\rho$ such that $P_{I_\rho}$ is free. Then it has an exponent identically equal to $0$ and so, for any closed subpolysegment $J\subset I$, $P_J$ admits a basis $\ee$ on which each $D_i$ acts via an upper triangular matrix $N_i$ with zero diagonals. Then, if we write an element $x$ of $V_J$ by $x=f_1e_1+\cdots+f_me_m$ with $f_i\in K_{J,n}$ and put $f:=(f_1,\dots,f_m)^T$, we have
$$ 0= D_i^m((e_{1},\dots,e_{m})f)=(e_{1},\dots,e_{m})\sum_{r=0}^m \binom{m}{r}N_{i}^r(t_i\partial_{t_i})^{m-r} f  $$
for any $1\leq i\leq n$, and this forces $f_1,\dots,f_m$ to be elements in $K$. So we see that $V_J$ is the $m$-dimensional $K$-vector space generated by $\ee$. Then, since $V_{I_\rho}=\varprojlim_{J\subset I_\rho} V_J$, $V_{I_\rho}$ is also an $m$-dimensional $K$-vector space. If $I_{\rho_1}\cap\cdots\cap I_{\rho_l}$ is non empty, the same is true for $V_{I_{\rho_1}\cap\cdots\cap I_{\rho_l}}$, and the restriction maps $V_{I_{\rho_1}\cap\cdots\cap I_{\rho_{l'}}}\to V_{I_{\rho_1}\cap\cdots\cap I_{\rho_l}}$ are isomorphisms. By definition, 
$$ V_I=\ker\left( \prod_{\rho\in I}V_{I_\rho}\rightrightarrows \prod_{\rho,\rho'\in I} V_{I_\rho\cap I_{\rho'}} \right). $$

On the other hand, Let $\mathscr{V}$ be the local system of $K$-vector spaces on the topological space $I\subset \RR^n_{>0}$ such that $\mathscr{V}_{I_\rho}$ is the constant local system $V_{I_\rho}$ and that the gluing is given by 
$$ (\mathscr{V}|_{I_\rho})|_{I_\rho\cap I_{\rho'}}= V_{I_\rho}|_{I_\rho\cap I_{\rho'}}\cong V_{I_\rho\cap I_{\rho'}}\cong V_{I_{\rho'}}|_{I_\rho\cap I_{\rho'}}= (\mathscr{V}|_{I_{\rho'}})|_{I_\rho\cap I_{\rho'}}. $$
Then, $V_I=\mathrm{H}^0(I,\mathscr{V})$. On the other hand, since $I$ is contractible, $\mathscr{V}$ is necessarily isomorphic to the constant local system $K^m$ and so $V_I=\mathrm{H}^0(I,\mathscr{V})\cong\mathrm{H}^0(I,K^m)=K^m$ is a $m$-dimensional $K$-vector space such that the restriction maps $V_I\to V_{I_\rho}$ are isomorphisms. So a $K$-basis of $V_I$ gives rise to a basis of $P_{I_\rho}$ for any $\rho$, and so it gives a basis of $P$. So we are done.

\end{proof}

Then we can reprove Gachet's $p$-adic Fuchs theorem as follows:

\begin{cor}\label{padicfuchs}
Let $P$ be a finite projective differential module over an open polysegment $I$ in $\RR_{>0}^n$ satisfying the Robba condition. Furthermore we assume that $P$ has $p$-adic non-Liouville exponent differences. Then $P$ admits a basis on which the matrix of action of $D_i$ for $1\leq i\leq n$ has entries in $K$ whose eigenvalues represents an exponent of $P$. Consequently, $P$ admits a canonical decomposition
$$ P=\bigoplus_{\lambda\in (\ZP/\ZZ)^n} P_\lambda, $$
where $P_\lambda$ is free with exponent identically equal to a representative in $\ZZ_p^n$ of $\lambda$. In particular, $P$ is free, and is extended from some finite differential module over a polydisc for the derivations $t_i\partial_{t_i}$, $1\leq i\leq n$.
\end{cor}
\begin{proof}
Let $\AA_1,\dots,\AA_k$ be the partition of an exponent $A$ of $P$ into $\ZZ^n$-cosets. Then it is a Liouville partition. Take an element $a_i\in\AA_i$ and let $\BB_i$ be the multiset whose cardinality is equal to that of $\AA_i$ and whose elements are identically equal to $a_i$. Then we can easily see that each $\BB_i$ is equivalent to $\AA_i$, $A$ is equivalent to $B=\bigcup_{i=1}^k\BB_i$ and that $\BB_1,\dots,\BB_k$ is a Liouville partition of $B$. Since any exponent equivalent to $A$ is an exponent of $P$, $B$ is an exponent of $P$. Then by Corollary \ref{decom_all}, $P$ admits a decomposition $P=\bigoplus_{\lambda\in (\ZP/\ZZ)^n} P_\lambda$ satisfying the required condition except that each $P_\lambda$ is projective. For each $P_\lambda$, we have $P_\lambda=(P_\lambda\otimes M_\lambda^\vee)\otimes M_\lambda$, and $P_\lambda\otimes M_\lambda^\vee$ has exponent identically equal to $0$ by Lemma \ref{exp123}, and hence free by Proposition \ref{freeness}. Consequently $P_\lambda$ is free.
\end{proof}

Corollary \ref{padicfuchs} is possibly slightly stronger than the result of \cite{Gac}, because we do not know yet if any finite projective differential module on polyannuli is free, and \cite{Gac} only treated finite free differential modules.

\nocite{*}
\normalem
\bibliographystyle{alpha}

\bibliography{references}

\begin{thebibliography}{BGR84}

\bibitem[Ber12]{Ber}
Vladimir~G Berkovich.
\newblock {\em Spectral theory and analytic geometry over non-Archimedean
  fields, Mathematical Surveys and Monographs 33}.
\newblock American Mathematical Soc., 2012.

\bibitem[BGR84]{BGR}
Siegfried Bosch, Ulrich G{\"u}ntzer, and Reinhold Remmert.
\newblock {\em Non-archimedean analysis, Grundlehren der mathematischen
  Wissenschaften 261}.
\newblock Springer Berlin, 1984.

\bibitem[CD94]{CD}
Gilles Christol and Bernard Dwork.
\newblock Modules diff\'erentiels sur les couronnes.
\newblock {\em Annales de l'Institut Fourier}, 44(3):663--701, 1994.

\bibitem[CM97]{CM2}
G.~Christol and Z.~Mebkhout.
\newblock Sur le theoreme de l'indice des equations differentielles $p$-adiques
  {II}.
\newblock {\em Annals of Mathematics}, 146(2):345--410, 1997.

\bibitem[DGS16]{ItoG}
Bernard Dwork, Giovanni Gerotto, and Francis~J Sullivan.
\newblock {\em An Introduction to G-Functions, Annals of Mathematics Studies
  133}.
\newblock Princeton University Press, 2016.

\bibitem[Gac99]{Gac}
F.~Gachet.
\newblock Structure fuchsienne pour des modules diff\'erentiels sur une
  polycouronne ultram\'etrique.
\newblock {\em Rendiconti del Seminario Matematico della Universit\`a di
  Padova}, 102:157--218, 1999.

\bibitem[Ked08]{KedF}
Kiran~S Kedlaya.
\newblock Full faithfulness for overconvergent {$F$}-isocrystals.
\newblock In {\em Geometric aspects of Dwork theory}, pages 819--836. de
  Gruyter, 2008.

\bibitem[Ked10]{Ked1}
Kiran~S Kedlaya.
\newblock {\em $p$-adic Differential Equations, Cambridge Studies in Advanced
  Mathematics 125}.
\newblock Cambridge University Press, 2010.

\bibitem[Ked15]{Ked2}
Kiran~S Kedlaya.
\newblock Local and global structure of connections on nonarchimedean curves.
\newblock {\em Compositio Mathematica}, 151(6):1096--1156, 2015.

\bibitem[KS17]{KS}
Kiran~S. Kedlaya and Atsushi Shiho.
\newblock Corrigendum: Local and global structure of connections on
  nonarchimedean curves.
\newblock {\em Compositio Mathematica}, 153(12):2658–2665, 2017.

\bibitem[KX10]{KX1}
Kiran~S Kedlaya and Liang Xiao.
\newblock Differential modules on $p$-adic polyannuli.
\newblock {\em Journal of the Institute of Mathematics of Jussieu},
  9(1):155--201, 2010.

\bibitem[Lan93]{La}
Serge Lang.
\newblock {\em Algebra, third edition}.
\newblock Addison-Wesley, 1993.

\bibitem[Lü77]{Lut}
W.~Lütkebohmert.
\newblock Vektorraumbündel über nichtarchimedischen holomorphen räumen.
\newblock {\em Mathematische Zeitschrift}, 152:127--144, 1976/77.

\end{thebibliography}

\end{document}